\documentclass[11pt,leqno]{amsart}
\usepackage[utf8]{inputenc}
\usepackage{graphicx}
\usepackage[english]{babel}
\usepackage{amsmath,amssymb}
\usepackage[bookmarks=true]{hyperref}
\usepackage[dvipsnames]{xcolor}
\hypersetup{colorlinks=true, linkcolor=NavyBlue, citecolor=Green}
\usepackage{mathrsfs}
\usepackage{enumerate,float, xfrac}
\usepackage{enumitem}
\usepackage{mathtools}
\usepackage{xcolor}
\usepackage{comment}
\usepackage{thm-restate}

\newtheorem{theorem}{Theorem}[section]
\newtheorem{lemma}[theorem]{Lemma}
\newtheorem{proposition}[theorem]{Proposition}
\newtheorem{corollary}[theorem]{Corollary}
\newtheorem{question}[theorem]{Question}
\theoremstyle{definition}
\newtheorem{definition}[theorem]{Definition}
\newtheorem{example}[theorem]{Example}
\newtheorem{remark}[theorem]{Remark}

\newtheorem*{remark*}{Remark}

\def\N{\mathbb{N}}
\def\Z{\mathbb{Z}}
\def\Q{\mathbb{Q}}

\title[On the measures of definable sets in profinite groups]{On the measure of definable sets \\ in profinite groups}
\author[M. Conte]{Martina Conte}
\author[M. Vannacci]{Matteo Vannacci}
\address{
Fakult\"at f\"ur Mathematik,
Universit\"at Bielefeld,
33501 Bielefeld, Germany}
\email{mconte@math.uni-bielefeld.de}
\address{
Dipartimento di Matematica `Ulisse Dini', Universit\`a degli Studi di Firenze, Viale Morgagni 67/a, 50134 Firenze, Italy}
\email{matteo.vannacci@unifi.it}



\newcommand{\comm}[1]{}

\begin{document}

\begin{abstract}
    We study the (Haar) measure of verbal sets in profinite groups, and especially in free profinite groups. Relating the measure of a verbal set with the width of the corresponding word, we find verbal sets of positive measure different from one. We prove a sufficient condition for the measure of a verbal set to be zero in a free profinite group of finite rank, from which we deduce that certain existentially definable sets in free profinite groups of finite rank can only have measure zero or one. Finally, we compute the measure of definable sets in torsion-free finitely generated abelian groups.  
\end{abstract}

\maketitle

\section{Introduction}

\subsection{State of the art}
Let $F_k$ be the abstract free group on $k$ generators. For a word $w\in F_k$ and a profinite group $G$ we study the ``size'' (in a suitable sense discussed below) of the set
\[
    G^{\{w\}}= \{ g\in G \mid \exists \boldsymbol{h}\in G^k:\ g=w(\boldsymbol{h})^{\pm 1}\}
\]

\noindent of \emph{$w$-values in $G$}.

Every profinite group $G$, being a compact topological group, can be seen as a probability space with respect to the unique Haar measure $\mu$ on $G$ such that $\mu(G)=1$. The use of probability in Group Theory is a very active area of research. Particular attention has been devoted to the study of probabilistic identities, where a word $w\in F_k$ is said to be a \emph{probabilistic identity} in a profinite group $G$ if $\mu(\{\boldsymbol{g}\in G^k \mid w(\boldsymbol{g})=1\})>0$. Several recent results in this direction can be found in \cite{KOTVW} and \cite{dlHTV}.

Here we take interest in a `dual' problem: given a word $w\in F_k$, we look at the (Haar) measure of the image of the word map $w:G^k\to G$. 
The study of images of word maps in finite non-abelian simple groups is well-developed. Just to mention one of the most celebrated results in this line of investigation, in \cite{LST11} it is shown that for any word $w$ there is a non-negative integer $N(w)$ such that, for all finite non-abelian simple groups $G$ with $\lvert G\rvert > N(w)$, the concatenation $w\ast w$ (see Section \ref{sec: preliminaries, width} for the definition) is surjective on $G$. Moreover, Ore's conjecture, proved in \cite{LOST}, states that the commutator map is surjective in any finite non-abelian simple group. Moving on to pro-$p$ groups, in \cite{AGKS} it is shown that, for any semisimple $\mathbb{Q}$-algebraic group $\mathbb{G}$, the image of any word map in $\mathbb{G}(\mathbb{Z}_p)$ is open, and hence of positive Haar measure.

In this work, we start a systematic study of the measure of images of word maps in pro-$p$ groups, and especially in free pro-$p$ groups. More generally, we consider (existentially) definable subsets (in the sense of model theory) of profinite groups, which naturally generalise verbal sets, and investigate their measure (see Section \ref{sec: preliminaries on formulae} for the definitions).  

Our main object of investigation will be the $\mathcal{F}$-spectrum of a profinite group, which we define in the following way. If $\mathcal{F}$ is a family of (first-order) formulae and $G$ is a profinite group endowed with its Haar measure $\mu=\mu_G$, we define the \emph{the $\mathcal{F}$-spectrum of $G$} to be
$$\mathrm{Spec}_\mu(\mathcal{F}, G) = \{ \mu(G_\varphi) \mid \varphi\in \mathcal{F} \}.$$

\subsection{Main results}

A concept that turns out to be useful in the study of the measure of verbal sets in profinite groups is that of \emph{width} of a word (see Section \ref{sec: preliminaries, width}).
It is not difficult to see that, using concatenations of words of finite width, one can easily produce words whose image has positive measure (see Remark \ref{rmk: concatenation of word of finite width has positive measure}). In particular, we find words that have image of measure zero in the free pro-$p$ group $H_d$, but whose concatenation has image of positive measure smaller than one in $H_d$ (compare with Corollary \ref{cor: in free pro-p groups there is a set with measure 0 that generates open subgroup= verbal set of concatenation}).
Therefore, it is natural to ask which words have image of measure zero. One can find such words using the arguments from \cite{JZ}. In fact, it is shown in \cite[Rmk.~4.4]{JZ} that the word $x^p [y,z]$ has image of measure zero over any free pro-$p$ group $H_d$ of rank $d\ge 4$.

Our first result is a generalisation of this observation, and it sheds some light on the situation of words whose number of variables is small compared to the rank of the free pro-$p$ group. In order to state it, we denote by $G^{\lbrace w+\rbrace}$ the image of a word $w$ in a group $G$, i.e., 
$$G^{\lbrace w+\rbrace} :=\lbrace g\in G\mid \exists\ \boldsymbol{h}\in G^k: g=w(\boldsymbol{h})\rbrace.$$ It is clear that $\mu(G^{\lbrace w\rbrace})=0$ if and only if $\mu(G^{\lbrace w+\rbrace})=0$ (see \cite[Chapter~21, Exercise~2]{FJ}).

\begin{restatable}{theorem}{propABC}
\label{prop:A}
    Let $w \in F_k$ be a word in $k$ variables and let $H_d$ be the free pro-$p$ group of rank $d$. Suppose that the image of the word map $w:H_d^k\to H_d$ is not surjective. Then, the image $H^{\{w+\}}_d$ of $w$ has measure zero in $H_d$ for all $d \ge k+1$. 
\end{restatable}

A similar result holds also for free profinite groups (Corollary \ref{cor: prop A for free profinite}).

It follows from Theorem \ref{prop:A} that the image of the word $w(x)=x^p$ has measure zero in any (non-abelian) free pro-$p$ group. However, we show in Appendix \ref{sec: family of p-gps where x^p has positive measure} that the image of the same word has positive measure in the families of finite $p$-groups $W_k:=C_p \wr C_{p^k}$ and $G_k:=C_p\wr W_k$, for $k\ge 1$. The latter even yields a family of finite $p$-groups where the image of $x^p$ gets smaller and smaller, thus providing an alternative proof of Theorem \ref{prop:A} for this specific word and the free pro-$p$ group of rank $3$. 

The simplest word to consider next is $w(x,y,z)=x^p [y,z]$. As this is a word in three variables, Theorem \ref{prop:A} states that the image of $w$ has measure zero in free pro-$p$ groups of rank at least $4$. Unfortunately, we can show in Appendix \ref{sec: family of p-gps where x^p has positive measure} that the measure of the word $w$ is large in the family $C_p \wr (C_p \wr C_{p^k})$ ($k\ge 1$). 

\begin{question}
    Is $\mu(\mathrm{im}(x^p [y,z]))=0$ in the free pro-$p$ groups of rank $2$ or~$3$?
\end{question}

Note that it is easy to see that the word $w(x,y,z)=x^p [y,z]$ satisfies 
\[
  w(H_2) = \mathrm{im}(x^p [y,z][u,v]),
\]
i.e.\ the word $w$ has width $2$ in the free pro-$p$ group of rank $2$.

Going further in our investigation, we consider the measure of sets defined by existential formulae $\eta_w$ of the form $\exists\boldsymbol{y}: w(x,\boldsymbol{y})=1$, which naturally generalise verbal sets. Drawing from Theorem \ref{prop:A}, we find conditions on the word $w$ that ensure that the set defined by $\eta_w$ in the free pro-$p$ group $H_d$ has either measure zero or one. We call such words \emph{$(p,d)$-zero-one words} (see Definition \ref{def: 0-1 word}), and we obtain the following.

\begin{theorem}\label{thm:0-1_intro}
    \label{prop: measure spectrum of E-formulae in free pro-p groups}
    Let $H_d=H_{p,d}$ be the free pro-$p$ group of rank $d\ge 2$. 
    Consider the family $\mathcal{E}(p)$ containing formulae $\eta_w(x)$ defined by $\exists\boldsymbol{y}: w(x,\boldsymbol{y})=1$, where $w$ runs over the set of $(p,d)$-zero-one words.
    Then $\mathrm{Spec}_\mu(\mathcal{E}(p); H_d)=\lbrace 0,1\rbrace$.  
\end{theorem}

Theorem~\ref{prop: measure spectrum of E-formulae in free pro-p groups} will be proved in Section~\ref{sec:thmA}.
With a definition of $d$-zero-one words independent of the prime $p$ (see Definition~\ref{def: 0-1 word}), we obtain:

\begin{corollary}
    \label{cor: measure spectrum of E-formulae in free profinite groups}
    Let $\widehat{F}_d$ be the free profinite group of rank $d\ge 2$. 
    Consider the family $\mathcal{E}$ containing formulae $\eta_w(x)$ defined by $\exists\boldsymbol{y}: w(x,\boldsymbol{y})=1$, where $w$ runs over the set of $d$-zero-one words.
    Then $\mathrm{Spec}_\mu(\mathcal{E}; \widehat{F}_d)=\lbrace 0,1\rbrace$. 
\end{corollary}

The proof of Theorem \ref{thm:0-1_intro} relies on understanding the conditions under which a word is surjective on a free pro-$p$ group $H_d$ (see Lemma \ref{lem: surjectivity conditions for words in free pro-p}); these feed into the definition of $(p,d)$-zero one words. Similar conditions on the surjectivity of words in free profinite groups (or more generally, on the words $w$ such that $(\widehat{F}_d)_{\eta_w}=\widehat{F}_d$) are listed in Lemma \ref{lem: n neq 0 and l does not divide n; measure is not 1}. With these one can show that, in analogy with discrete free groups, the sentence $\forall x:\eta_w(x)$ is true in $\widehat{F}_d$ if and only if it is true in every group (Corollary \ref{cor: trivial eq in free profinite groups}).

In contrast to free pro-$p$ and free profinite groups, the $\mathcal{F}$-spectrum of finitely generated abelian groups is very far from displaying a $0-1$ behaviour. In Proposition \ref{prop: measure of def sets in abelian groups} we give a formula for the measure of any definable subset of a torsion-free finitely generated abelian group.  
Finally, in Appendix \ref{sec: appendix A further results} we collect some further results on $\mathcal{F}$-spectra. For instance, we prove in Lemma \ref{lem: definable sets in count free have measure 0-1} that in the free profinite group of \emph{countable} rank the measure of any verbal set is either zero or one, which shows that finite generation is necessary to find verbal sets of positive measure different from one in free pro-$p$ and free profinite groups.

\subsection*{Notational conventions} 
If $G$ is a finitely generated profinite group, we denote by $\mathrm{d}(G)$ its minimal number of generators. The Frattini subgroup of $G$ will be denoted by $\Phi(G)$. 
If $N$ is a normal subgroup of finite index (respectively an open normal subgroup) in a profinite group $G$ we write $N\lhd_f G$ (respectively $N\lhd_o G$).

The abstract free group of rank $d$ will be denoted $F_d$. 
Following the notation in \cite{JZ}, unless otherwise stated, given a prime $p$ we denote by $H_d$ the free pro-$p$ group of rank $d$. If we need to specify the prime $p$ we will write $H_{p,d}$. The free profinite group of rank $d$ will be denoted $\widehat{F}_d$.

If $G$ is a group and $w$ is a word, $w(G)$ will denote the verbal subgroup generated by $w$, i.e. $w(G)=\langle G^{\lbrace w\rbrace}\rangle=\langle G^{\lbrace w+\rbrace}\rangle$. 

Given a word $w$, we denote by $\#\mathrm{var}(w)$ the number of variables occurring in $w$.
We will say that $w$ is a \emph{non-commutator word} if the sum of the exponents of at least one variable occurring in $w$ is non-zero, and we will call $w$ a \emph{commutator word} otherwise. Equivalently, a word $w\in F_k$ is a non-commutator word if and only if it does not belong to the commutator subgroup $[F_k,F_k]$. 

Finally, if $a$ and $b$ are integers and $a$ is non-zero, we write $a\mid b$ to say that $a$ divides $b$ and $a\nmid b$ ta say that $a$ does not divide $b$. For any two integers $c,d$ we write $\gcd(c,d)$ to denote the greatest common divisor of $c$ and $d$. If $\lbrace d_i\mid i\in I\rbrace$ is a finite set of integers where $I\subseteq \N$ is clear from the context, we will shorten $\gcd(d_i\mid i\in I)$ with $\gcd_i(d_i)$.

\subsection*{Acknowledgments} 
The authors would like to thank Steffen Kionke for interesting discussions and questions during the preparation of this article. The first named author was funded by the Deutsche Forschungsgemeinschaft (DFG, German Research Foundation) – Project-ID~491392403 – TRR~358. The second named author is funded by the Italian program Rita Levi Montalcini for young researchers (Edition 2021).

\newcommand{\noshow}[1]{}

\section{Preliminaries}

\subsection{Measurable sets}

Given a profinite group $G$, there exists a unique Haar measure $\mu$ defined on the Borel $\sigma$-algebra of open sets of $G$, normalised so that $\mu(G)=1$ (\cite[Proposition 21.2.1]{FJ}). This measure is countably additive and translation invariant. If $G$ is a finite group, the Haar measure coincides to the usual counting probability measure.

If $S$ is a closed subset of $G$, then its Haar measure is given by (\cite[Section 11.1, page 206]{LuSe})
\begin{equation}
\label{eq: measure r.f. group}
\mu(S)=\inf_{N\lhd_o G}\frac{\vert SN/N \vert}{\vert G/N\vert }.
                  \end{equation}

We collect here some useful facts about the Haar measure on profinite groups that we will need later on. For more theoretical background see \cite[Chapter 21]{FJ}.

A fundamental result for computing measures in profinite groups is the following.

\begin{lemma}[\cite{FJ}, Lemma 21.1.1 (a)]
Let $G$ be a profinite group and let $H$ be a closed subgroup of $G$. Then $\mu(H)=\frac{1}{[G:H]}$.
\end{lemma}

From this lemma and the translation invariance of the Haar measure it immediately follows:

\begin{corollary}[\cite{FJ}, Lemma 21.1.1 (c)]
    Let $U$ be an open subset of a profinite group $G$. Then $\mu(U)>0$.
\end{corollary}

The following result relates the measure of a set of a profinite group $G$ with the measure of its image in a quotient of $G$.

\begin{proposition}[\cite{FJ}, Proposition 21.2.2]
\label{prop: quotient measure of a set}
Let $\pi\colon G\rightarrow H$ be an epimorphism of profinite groups and let $\mu_G$ and $\mu_H$ be the Haar measures on $G$ and $H$ respectively. Then $\mu_H(B)=\mu_G(\pi^{-1}(B))$ for every measurable set $B$ of $H$. 
\end{proposition}

\begin{corollary}
\label{cor: measure grows with quotient}
    Let $\pi\colon G\rightarrow H$ be an epimorphism of profinite groups and let $A$ be a measurable set of $G$ such that $\pi(A)$ is measurable in $H$ and $\mu_H(\pi(A))=0$. Then $\mu_G(A)=0$.
\end{corollary}

\begin{proof}
    Observe that $A\subseteq \pi^{-1}(\pi(A))$ and use Proposition \ref{prop: quotient measure of a set}.
\end{proof}

Finally, we will use the following basic observation.

\begin{lemma}
    \label{lem: measure closed set =1 iff C=G}
    Let $G$ be a profinite group and let $C$ be a closed subset of $G$. Then $\mu_G(C)=1$ if and only if $C=G$.
\end{lemma}

    \begin{proof}
        Suppose that $\mu_G(C)=1$. Then $\mu_G(G\setminus C)=0$ and $G\setminus C$ is open, which forces $G\setminus C=\emptyset$ and so $C=G$.
    \end{proof}

\subsection{Width}
\label{sec: preliminaries, width}

Recall from the Introduction that, given a word $w$ and a group $G$, we denote by $G^{\{w\}}$ the set of $w$-values in $G$, i.~e., $G^{\{w\}}=\{ g\in G \mid \exists \boldsymbol{h}\in G^k:\ g=w(\boldsymbol{h})^{\pm 1}\}$.

A word $w\in F_k$ in some abstract free group $F_k$ has \emph{finite width} in a group $G$ if there exists $m\in \mathbb{N}$ such that each element of the verbal group $w(G)=\langle G^{\lbrace w\rbrace}\rangle$ can be written as a product of at most $m$ $w$-values.
If $G$ is a profinite group, this is equivalent to the verbal subgroup $w(G)$ being closed in $G$ (\cite[Proposition 4.1.2]{S}).
If a word $w$ has finite width in a group $G$, we call the minimal $m\in\N$ such that every element of $w(G)$ can be written as a product of at most $m$ $w$-values \emph{the width} of $w$ in $G$.

By \cite[Theorem~1.2]{JZ}, if a word $w\in F_k$ has infinite width in a free pro-$p$ group $H$ of finite rank, then $w(H) \le (H')^p H''$, and hence, as the latter is a subgroup of infinite index in $H$, we can deduce that $\mu(H^{\{w\}})=0$. Note that the same argument applies to words that have infinite width in any finitely generated pro-$p$ group $G$ such that $(G')^pG''$ has infinite index in $G$.

Therefore we can restrict our attention to words of finite width, i.e.\ those not belonging to $(F')^p F''$.
We will now show that, starting from words of finite width, we can construct words whose images have positive Haar measure. We first need some notation.

 For each positive integer $i$, let $\varepsilon_i\in\lbrace 1, - 1\rbrace$ and set $\boldsymbol{\varepsilon}$ to be the sequence $(\varepsilon_i)_{i\in \mathbb N}$. Given any positive integer $m$, let $\boldsymbol{\varepsilon}_m:=(\varepsilon_i)_{i=1}^m$ be the $m$-tuple whose entries are the first $m$ entries of $\boldsymbol{\varepsilon}$. Define the word $w^{\ast n, \boldsymbol{\varepsilon}}$ for every $n\geq 1$ inductively as follows: 
\begin{align*}
w^{\ast 1, \boldsymbol{\varepsilon}}(\boldsymbol{x}_1) &= w(\boldsymbol{x}_1)^{\varepsilon_1}\\  w^{\ast (n+1), \boldsymbol{\varepsilon}}(\boldsymbol{x}_1,\ldots,\boldsymbol{x}_{n+1}) &= w^{\ast n, \boldsymbol{\varepsilon}}(\boldsymbol{x}_1,\ldots, \boldsymbol{x}_n) \cdot w(\boldsymbol{x}_{n+1})^{\varepsilon_{n+1}},
\end{align*}
where $\boldsymbol{x}_1,\ldots,\boldsymbol{x}_{n+1}$ are $k$-tuples of variables such that $\boldsymbol{x}_i$ is disjoint from $\boldsymbol{x}_j$ for every $i\neq j$.

For simplicity of notation, if all exponents $\varepsilon_i$ are equal to $1$, we will denote
$$w^{\ast n}(\boldsymbol{x}_1,\ldots, \boldsymbol{x}_n):=w^{\ast n, (1,\ldots,1)}(\boldsymbol{x}_1,\ldots, \boldsymbol{x}_n). $$

The following lemma follows immediately from the definitions.

\begin{lemma}
\label{lem: if finite width, verbal subgroup = union of verbal sets}
    Let $w\in F_k$ be a word in $k$ variables with finite width $m$ in a group $G$. Then 
    $$w(G) = \bigcup_{(\varepsilon_1,\ldots,\varepsilon_m)\in\lbrace 1, -1\rbrace^{m}} G^{\lbrace w^{\ast m, (\varepsilon_1,\ldots,\varepsilon_m)}\rbrace}.$$
\end{lemma}

As a consequence, we readily get the following.

\begin{corollary}
\label{rmk: concatenation of word of finite width has positive measure}
Consider $w\in F_k$ a word in $k$ variables of finite width $m$ in a profinite group $G$.  
Suppose that $w(G)$ has finite index in $G$. Then, the image of at least one of the words in $km$ disjoint variables 
\begin{equation*}\label{eq:concatenation}
    w^{\ast m, (\varepsilon_1,\ldots,\varepsilon_m)}(x_1,\ldots,x_{km}) = w^{\varepsilon_1}(x_1,\ldots,x_k) \cdots w^{\varepsilon_m}(x_{k(m-1)+1},\ldots,x_{km})
\end{equation*}
has positive Haar measure. 
\end{corollary}

\begin{proof}
    By Lemma \ref{lem: if finite width, verbal subgroup = union of verbal sets} the union of the images of the words $w^{\ast m, (\varepsilon_1,\ldots,\varepsilon_m)}$ coincides with the verbal subgroup $w(G)$, and therefore has positive Haar measure, given by the index of $w(G)$ in $G$. It follows that at least one of the sets in the union has positive Haar measure.
\end{proof}

\begin{example}
Let $G$ be a finitely generated pro-$p$ group with minimal number of generators $d:=\mathrm{d}(G)$, and consider the word $w(x,y,z)=x^p [y,z]$. Since the commutator word has width $d$ in $G$ \cite[Lem.~1.23]{DdSMS99}, the word $w$ has also width $d$ in $G$. Moreover, in this case  
$$ w^{\ast d, \boldsymbol{\varepsilon}}= w^{\ast d, (1,\ldots,1)}=w^{\ast d}$$
for every choice of the $d$-tuple $ \boldsymbol{\varepsilon}$.
Since the Frattini subgroup $\Phi(G)$ of $G$ is the open subgroup given by
$$\Phi(G)=G^p[G, G]= G^{\lbrace w^{\ast d}\rbrace},$$
the image of the word $w^{\ast d}$ has positive Haar measure in $G$, given by the index $[G:\Phi(G)]=p^{-d}$.
\end{example}

\subsection{Some observations on surjective word maps in pro-\texorpdfstring{$p$}{p} groups}

The following lemma is probably well-known to experts, but, since we could not find a satisfactory reference for our purposes, we include it here for the reader's convenience.

\begin{lemma}\label{lem:non-surj} 
Let $w\in F_k$ be a word in $k$ variables and let $G$ be a pro-$p$ group.
\begin{enumerate}
    \item The word map $w: G^k\to G$ is surjective if and only if the word map $\overline{w}: (G/\Phi(G))^k \to G/\Phi(G)$ is surjective.
    \item \label{item: im non-surj is in frattini} The image of a non-surjective word map $w$ is contained in $\Phi(G)$. 
    \item \label{item: lemma non-surj non-surj word is id in elem ab p-gps} Any non-surjective word is an identity in any elementary abelian $p$-group. 
    \item Any element of $\Phi(G)$ is the image of some non-surjective word map.
\end{enumerate}
\begin{proof}

\begin{enumerate}

   \item If the word map $w$ is surjective on $G$, it is clear that $\overline{w}$ is surjective on the Frattini quotient. On the other hand, suppose that $\overline{w}$ is surjective on $G/\Phi(G)$, and write
   \[ 
     \overline{w}(\boldsymbol{g}\Phi(G)) = ( e_1 g_1 + \ldots + e_d g_d  ) \Phi(G)
   \]
   where the $e_i$ are the exponent sums of the variables modulo $p$. Since $\overline{w}$ is surjective, without loss of generality we can suppose that $e_1 \not\equiv 0 \mod p$. Hence, $e_1$ is invertible modulo $p$ and therefore the word $w(x_1,1,\ldots,1)=x_1^{e_1}$ is surjective on $G$. 
   
   \item \label{item: 2} For a word map to be non-surjective on $G/\Phi(G)$, we must have that all exponent sums are multiples of $p$, hence the image is in $\Phi(G)$. 
   
   \item This follows from the proof of \eqref{item: im non-surj is in frattini}.

   \item Finally, since the word $x^p[y,z]$ has finite width in $G$, say $m$, any element of $\Phi(G)$ is in the image of the word $(x^p[y,z])^{\ast m}$. 
   \end{enumerate}
   \end{proof}
\end{lemma}

  \begin{corollary}
  \label{cor: not surj on free pro-p is not surj on pro-p}
        Let $w\in F_k$ be a word that is not surjective on the free non-abelian pro-$p$ group $H_d$ of rank $d\geq 2$. Then $w$ is not surjective on any $d$-generated pro-$p$ group. In particular, $\mu_{H_d}(H_d^{\lbrace w+\rbrace})=1$ if and only if $\mu_{G}(G^{\lbrace w+\rbrace})=1$ for some (and hence for every) pro-$p$ group $G$. 
 \end{corollary}
\begin{proof}
    By Lemma \ref{lem:non-surj}, \eqref{item: im non-surj is in frattini}, the image of $w: H_d^k\rightarrow H_d$ is contained in $\Phi(H_d)$. Given any $d$-generated pro-$p$ group $G$, the equality of sets $G^{\lbrace w+\rbrace}=\pi(H_d^{\lbrace w+\rbrace}K)$ holds, where $\pi:H_d\rightarrow G$ is the natural projection and $K$ denotes its kernel. Similarly, $\Phi(G)=\pi( \Phi(H_d)K)$. In particular, if $H_d^{\lbrace w+\rbrace}\subseteq \Phi(H_d)$, then $G^{\lbrace w+\rbrace}\subseteq \Phi(G)$, and $w$ is not surjective in $G$.
\end{proof}

Even if not directly relevant for our work, we observe that Lemma \ref{lem:non-surj} can be used to give an alternative proof that there exist words with image of positive measure in any finitely generated pro-$p$ group. 
In fact, we know that
\[
   \Phi(G)= \bigcup 
   G^{\{w\}}
\]
where the union is taken over the non-surjective words $w$.
By Baire category theorem \cite[IX, \S 5, pp.\ 190--195]{bou}, there must be a word $w$ whose image has non-empty interior, hence, positive measure.

\subsection{Definable sets as measurable sets}
\label{sec: preliminaries on formulae}

Recall that the language of groups is given by $\mathcal{L}_{\text{gp}}=\lbrace 1, \cdot, ^{-1}\rbrace.$ For our purposes we regard a first-order formula in the language of groups as a finite string of symbols built using the symbols of $\mathcal{L}_{\text{gp}}$, variable symbols, the equality symbol ($=$), the connectives ``and'' ($\wedge$), ``or'' ($\vee$), ``not'' ($\neg$), the existential quantifier ($\exists$), the universal quantifier ($\forall$) and parentheses.

If the existential quantifier is the only quantifier occurring in a formula $\varphi$, one says that $\varphi$ is an \emph{existential formula}.
A variable in a formula is said to be \textit{free} if it is not bound to any quantifier.  

We will simply call \emph{formula} a first-order formula in the language of groups with one free variable and without parameters (i.e., when evaluating formulae in a group $G$, no constants from $G$ are allowed). If $\varphi$ is a formula and $x$ is a free variable, we write $\varphi(x)$.
Therefore, a formula $\varphi(x)$ is equivalent to a formula given by the following expression
\begin{equation}
\label{eq: general form of a formula}
\mathbf{Q}\ \boldsymbol{y}:\bigvee_{j=1}^m \left(w_{j, 1}(x,\boldsymbol{y})\#_{j,1} 1\wedge\cdots\wedge w_{j,k_j}(x,\boldsymbol{y})\#_{j,k_j}1\right),
\end{equation}
where $\mathbf{Q}$ is any string of quantifiers in $\lbrace\forall,\exists\rbrace$, every $w_{j,i}$ is a word, $m$ and the $k_j$'s are non-negative integers, and $\#_{j,i}\in \lbrace =, \neq\rbrace$.

Given a group $G$, the set defined by a formula $\varphi$ is given by 
$$G_\varphi:=\lbrace g\in G\mid \varphi(g) \ \text{is true} \rbrace.$$ 

If a set $S\subseteq G$ can be defined by an existential formula, we will say that $S$ is an \emph{existentially definable} (in symbols, $\exists$-definable) set.
Note that both $G^{\lbrace w\rbrace}$ and $G^{\lbrace w+\rbrace}$ are clearly $\exists$-definable sets.

From (\ref{eq: general form of a formula}) it follows that any existential formula is equivalent to a formula of the form
\begin{equation}
\label{eq: general form of an existential formula}
\exists \boldsymbol{y}:\bigvee_{j=1}^m \left(\bigwedge_{i=1}^{\ell_j} w_{j,i}(x,\boldsymbol{y})=1\wedge\bigwedge_{i=\ell_j+1}^{k_j}  w_{j,i}(x,\boldsymbol{y})\neq 1\right), 
\end{equation}

where $m$ and the $k_j$'s are positive integers, the $\ell_j$'s are non-negative integers and the $w_{j,i}$'s are words.

A special case of existential formulae is given by \emph{positive boolean combinations of words}.
Let $\lbrace w_{j}\mid j=1,\ldots, m \rbrace$ be a collection of words on disjoint sets of variables, where $m$ is a positive integer. A boolean combination $\mathbf{b}$ of these words is a boolean combination of the formulae $\exists\boldsymbol{y}: x=w_{j,i}(\boldsymbol{y})$. 

For notational convenience, given a word $w$ in $k+1$ variables, we denote by $\eta_w$ the formula
$$\eta_w:= \exists \boldsymbol{y}: w(x,\boldsymbol{y})=1,$$
and call it the \emph{equation (formula) related to $w$}.
If $w$ is clear from the context we will omit it, and, if $w=w_{j,i}$, we will also write $\eta_{j,i}$ for the corresponding formula. Similarly, we denote by $\iota_w$ (respectively $\iota$, or $\iota_{j,i}$) the formula 
$$\iota_w:= \exists \boldsymbol{y}: w(x,\boldsymbol{y})\neq 1,$$
and call it the \emph{inequality (formula) related to $w$}.

\begin{lemma}
    \label{lem: G_eta closed, G_iota open}
 Let $\eta$ (respectively~$\iota$) be an equation (respectively~inequality) as defined above, and let $G$ be a profinite group. Then $G_\eta$ is closed and $G_\iota$ is open. In particular, $G_\eta$ and $G_\iota$ are measurable.
\end{lemma}

\begin{proof}
    Let $w$ be the word in $k+1$ variables associated to the equation $\eta$. Then $w$ is a continuous map 
    $$w\colon G\times G^k\rightarrow G.$$
    It follows that $S:=\lbrace (x,\boldsymbol{y})\mid w(x,\boldsymbol{y})=1\rbrace$ is closed in $G$. Now the projection onto the first factor $\pi_G\colon G\times G^k\rightarrow G$ that sends $(x,\boldsymbol{y})$ to $x$ is an open and closed map. It follows that $G_\eta=\pi_G(S)$ is closed. 

    Similarly, for any inequality $\iota$, the set $S':=\lbrace (x,\boldsymbol{y})\mid w(x,\boldsymbol{y})\neq 1\rbrace$ is open in $G$. Therefore $G_\iota=\pi_G(S')$ is open.
\end{proof}

Recall from the Introduction that, if $\mathcal{F}$ is a family of formulae (e.g.~the family of existential formulae), and $G$ is a profinite group endowed with the Haar measure $\mu=\mu_G$, we write
$$\mathrm{Spec}_\mu(\mathcal{F}, G) = \{ \mu(G_\varphi) \mid \varphi\in \mathcal{F} \}$$ and we call this set the \emph{$\mathcal{F}$-spectrum of $G$}; this is the set of all real numbers between 0 and 1 that can occur as Haar measure of sets defined by the formulae $\varphi\in\mathcal{F}$ in $G$.
More generally, we denote by $\mathrm{Spec}_\mu(\mathcal{F}, \mathscr{C})$ the collection of possible measures that sets defined by formulae in $\mathcal{F}$ can have in groups belonging to a certain class $\mathscr{C}$.

\section{Words with small image in profinite groups}

In this section we deal with a particular type of existential formulae: images of word maps. The results of this section will be used in Section~\ref{sec: existential}.

\subsection{Words with image of measure 0 in free pro-\texorpdfstring{$p$}{p} groups}

In this section we prove Theorem \ref{prop:A}. We start by recalling a result from \cite{JZ}.

For any pro-$p$ group $G$, let $D_n(G)$ be the $n$-th \emph{dimension subgroup} of $G$ defined in the following way (see \cite[Chapter~11]{DdSMS99}): $D_1(G):= G$ and, for every $n>1$, $$D_n(G):= D_{\lceil\frac{n}{p}\rceil}^p \prod_{i+j=n}{[D_i, D_j]}.$$

It follows from the definition that $[D_i(G), D_j(G)]\leq D_{i+j}(G)$ for all positive integers $i,j$, and that $D_{i}^p(G)\leq D_{pi}(G)$. In particular, if $G$ is finitely generated, for every $n$ one has 
\begin{equation}
\label{eq: containment frattini dim subgroup}
\Phi(D_n(G))=[D_n(G), D_n(G)]D_n^p(G)\leq D_{2n}(G)D_{np}(G)\leq D_{n+1}(G).
\end{equation}

Now let $H_d$ be the free pro-$p$ group of rank $d$. For every positive integer $n$ we define the following
numbers:
\[
  a_n := \log_p \lvert D_n(H_d) : D_{n+1}(H_d)\rvert, \qquad b_n := \log_p \lvert H_d : D_{n+1}(H_d)\rvert.
\]

\begin{lemma}[{\cite[Lemma~4.3]{JZ}}]
 When $n$ tends to infinity the following holds:
$$a_n = \frac{d^n}{n} (1 + o(1)) \qquad \text{and} \qquad b_n = \frac{d^{n+1}}{
(d-1)n} (1 + o(1)).$$
\end{lemma}

We will also need the following well-known result, of which we include the proof for completeness.

\begin{lemma}
    \label{lem: set of lower box dim <1 has haar measure 0}
    Let $G$ be a pro-$p$ group with Haar measure $\mu_G$ and let $S$ be a closed subset of $G$ whose lower box dimension with respect to some filtration $\mathcal{G}=(G_n)_{n\in \mathbb N}$ of $G$ is smaller than 1; that is 
    \[
    \underline{\dim}_{\, \mathcal{G}}(S) := \liminf_{n\to \infty} \frac{\log_p \lvert S G_{n} : G_{n} \rvert}{ \log_p \lvert G:G_{n}  \rvert}<1.
    \]
    Then $\mu_G(S)=0$.
\end{lemma}

\begin{proof}
Let $\mathcal{G}=(G_n)_{n\in \mathbb N}$ be a filtration of $G$ such that $\underline{\dim}_{\, \mathcal{G}}(S)<1$.
Let $\varepsilon >0$ such that $d:= \underline{\dim}_{\, \mathcal{G}}(S)+\varepsilon<1$.
Then there exist a subsequence $(G_{n_m})_{m\in \mathbb N}$ of $\mathcal{G}$ and a natural number $m_0$ such that, for every $m\ge m_0$,
$$\log_p \lvert S G_{n_m} : G_{n_m} \rvert \leq d\log_p \lvert G:G_{n_m}  \rvert.$$

Therefore, for every $m\geq m_0$ we have

$$\frac{\lvert S G_{n_m} : G_{n_m} \rvert}{\lvert G : G_{n_m} \rvert}\leq \frac{\lvert G : G_{n_m} \rvert^d}{\lvert G : G_{n_m} \rvert}= \lvert G : G_{n_m} \rvert^{d-1}.$$

Since the right hand side tends to $0$ when $m\rightarrow \infty$, we conclude that 
$$\lim_{m\to\infty}{\frac{\lvert S G_{n_m} : G_{n_m} \rvert}{\lvert S:G_{n_m}  \rvert}}=0.$$
Then the result follows from (\ref{eq: measure r.f. group}) and the inequality
    $$\mu(S) = \inf_{N\lhd_o G}\frac{\vert SN/N \vert}{\vert G/N\vert }\le \lim_{m\to\infty}{\frac{\lvert S G_{n_m} : G_{n_m} \rvert}{\lvert S:G_{n_m}  \rvert}}. $$
\end{proof}

We can now prove the main result of this section.

\begin{proof}[Proof of Theorem~\ref{prop:A}]
We start by observing that, by part \eqref{item: lemma non-surj non-surj word is id in elem ab p-gps} of Lemma~\ref{lem:non-surj} and the containment \eqref{eq: containment frattini dim subgroup}, for any non-surjective word $w$ we have that $$w(D_n(H_d)) \subseteq D_{n+1}(H_d).$$ Therefore, since $[H_d, D_{n}(H_d)]\leq D_{n+1}(H_d)$, 
for all tuples $(\overline{g}_1,\ldots,\overline{g}_k)$ of elements belonging to $H_d/D_{n+1}(H_d)$ and all tuples $(\overline{\alpha}_1,\ldots,\overline{\alpha}_k)$ of elements in $D_n(H_d)/D_{n+1}(H_d)$, we have that
\[
   w(\overline{g_1 \alpha_1},\ldots, \overline{g_k \alpha_k}) = w(\overline{g}_1,\ldots, \overline{g}_k). 
\]
From this we conclude that 
\begin{align*}
  \log_p\lvert (H_d/D_{n+1}(H_d))^{\{w+\}} \rvert &\le \log_p\lvert H_d/D_n(H_d) \rvert^k = k\cdot b_{n-1}  \\ &= k\cdot \frac{d^n}{(d-1)(n-1)} (1+o(1)).  
\end{align*}
  
Hence, considering the lower box dimension $\underline{\dim}_{\mathcal{D}}$ with respect to the filtration $\mathcal{D}=(D_n(H_d))_{n\in \mathbb N}$ we obtain
\[
   \underline{\dim}_\mathcal{D}(H_d^{\{w+\}}) = \liminf_{n\to \infty} \frac{\log_p \lvert H_d^{\{w+\}} D_{n}(H_d) : D_{n}(H_d) \rvert}{ \log_p \lvert H_d:D_{n}(H_d)  \rvert} \le \frac{k}{d}.
\]
Then the result follows from Lemma \ref{lem: set of lower box dim <1 has haar measure 0}.
\end{proof}

\begin{remark}
\label{rmk: verbal sets have positive mesure in abelian groups}
    As we will examine more thoroughly in Section \ref{sec: measures in fg abelian groups}, where we consider the measures of definable sets of abelian groups, there is no analogous result to the previous Theorem \ref{prop:A} when $d=1$. Indeed, let $w(x_1,\ldots,x_k)$ be any non-trivial word and let $G$ be a finitely generated abelian group, say with minimal number of generators $d$. When evaluated in $G$, using additive notation, the previous word can be rewritten as $w(g_1,\ldots, g_k)=a_1g_1+\cdots+ a_kg_k$ for some integers $a_1,\ldots,a_k$. Without loss of generality, assume that $a:=a_1\neq 0$. Then $aG\subseteq G^{\lbrace w+\rbrace}$. From the fact that $aG$ is a subgroup of $G$ of index $\frac{1}{d^{a}}$, we obtain that $\mu(G^{\lbrace w+\rbrace})>0$.  
\end{remark}

Given any set $X$ and positive integer $m$, denote by $X^{\ast m}$ the set 
$$X^{\ast m}:=\lbrace x_1\cdots x_m\mid x_i\in X\ \text{for all}\ i=1,\ldots,m \rbrace.$$
From Theorem~\ref{prop:A} and Corollary \ref{rmk: concatenation of word of finite width has positive measure} we deduce the following.

\begin{corollary}
\label{cor: in free pro-p groups there is a set with measure 0 that generates open subgroup= verbal set of concatenation}
Let $H_d$ be the free pro-$p$ group of rank $d$. Then there exists a subset $X\subset H_d$ with $X^{-1}=X$, $\mu(X)=0$ and $0<\mu(X^{\ast m})<1$ for some natural number $m$. 
\end{corollary}

\begin{proof}
Take $X=H_d^{\{x^p\}}$. According to Theorem \ref{prop:A}, the measure of the image of $w(x)=x^p$ is zero.
        However, $w$ has finite width in $H_d$ and, together with the positive solution of the restricted Burnside problem, one gets that the verbal subgroup that $w$ generates is open (see \cite[Corollary 8.3]{W}).
\end{proof}

In view of the previous results, it is natural to ask the two following questions regarding the growth of the measure of verbal sets.

\begin{question}\label{quest:small_d}
    Let $w$ be a non-surjective word in $k$ variables. Suppose that there is a positive integer $d_0$ such that $\mu(H_{d_0}^{\lbrace w+\rbrace})>0$. We know that for any $d\geq k+1$ one has $\mu(H_{d}^{\lbrace w+\rbrace})=0$ (and hence $d_0\leq k$). What is the measure $\mu(H_{d}^{\lbrace w+\rbrace})$ for $d_0<d\leq k$?
\end{question}
For some considerations related to Question~\ref{quest:small_d} see the Appendix~\ref{sec: family of p-gps where x^p has positive measure}.
\begin{example}
For example, we know that, for $$w(x,y_1,y_2,u_1,u_2)=x^p[y_1,y_2][u_1,u_2],$$ $\mu(H_{2}^{\lbrace w+\rbrace})=\frac{1}{p^2}$ and $\mu(H_{d}^{\lbrace w+\rbrace})=0$ for every $d\geq 6$. Then the previous question asks  what the measure of the image of $w$ is in $H_d$ for $3\leq d<6$.
\end{example}

A related question is the following.

\begin{question}
    Let $H_d$ be a fixed free pro-$p$ group of rank $d$.
    Let $w$ be a non-surjective word in $k$ variables. 
    For any $m\geq 1$ consider the concatenations $w^{\ast m,\boldsymbol{\varepsilon}}$ of $w$. How does the measure of $H_d^{\lbrace w^{\ast m,\boldsymbol{\varepsilon}}+\rbrace}$ change with $m$?
\end{question}

\begin{example}
For instance, we know that, if $w(x)=x^p$, then $\mu(H_d^{\lbrace w+\rbrace})=0$ for any $d\geq 2$. However, for any $d$ the word $w$ has finite width $m_d$ in $H_d$, and its verbal subgroup is open. Therefore, with $d$ fixed, the image of the word $w^{\ast m_d}$ has positive measure in $H_d$. How does this measure change with growing $m$ such that $mk>d-1$?
\end{example}

From Theorem \ref{prop:A} we also obtain a lower bound on the width of a word $w$ in a free pro-$p$ group $H_{d}$, which depends on the rank $d$ and the number of variables of $w$.

\begin{corollary}
    Let $w$ be a word in $k$ variables that has finite width $m_{w,d}$ in the free pro-$p$ group $H_d$ of rank $d$, and assume that the verbal subgroup $w(H_d)$ is open in $H_d$. Then $m_{w,d}> \frac{d-1}{k}$. 
\end{corollary}

\begin{proof}
    Let $m:=m_{w,d}$.
    By Theorem \ref{prop:A}, if $d\geq mk+1$, then the image of each of the words in $mk$ variables $w^{\ast m, (\pm_1,\ldots,\pm_m)}$ has measure $0$ in $H_d$. Since by Corollary \ref{rmk: concatenation of word of finite width has positive measure} the measure of at least one of the images of these words is positive, we must have $m > \frac{d-1}{k}$.
\end{proof}

For instance, for the word $w(x,y,z)=x^p[y,z]$ we get $m_{w,d}>\frac{d-1}{3}$, and indeed it is well known that $m_{w,d}=d$. For the word $x^p$ we obtain the following.
\begin{corollary}
  The word $w(x)=x^p$ has width $m_{w,d}> d-1$ in the free pro-$p$ group $H_d$ of rank $d$.  
\end{corollary}

    \subsection{Words with image of measure 0 in free profinite groups}

In this section we establish an analogous result to Theorem \ref{prop:A} for free profinite groups.

    From Theorem \ref{prop:A} and Corollary \ref{cor: measure grows with quotient} we immediately obtain:

    \begin{corollary}
    \label{cor: prop A for free profinite}
    Let $d\ge 2$ and let $\widehat{F}_d$ be the free profinite group of rank $d$.
    Suppose that $w$ is a word in $k$ variables and that $p$ is a prime such that $w$ is not surjective in the free pro-$p$ group $H_{p,d}$. Then $w$ has measure zero in $\widehat{F}_d$ whenever $k+1\le d$.
    \end{corollary}

    \begin{remark}
        We will see later in Corollary \ref{cor: surjective words} and Lemma \ref{lem: surjectivity conditions for words in free pro-p} that, if a word $w$ is not surjective on a free non-abelian profinite group of finite rank, then there exists a prime $p$ such that $w$ is not surjective on $H_{p,d}$.
         Therefore, the condition on the non-surjectivity of $w$ in Corollary \ref{cor: prop A for free profinite} is always verified if $w$ is not surjective on $\widehat{F}_d$. 
    \end{remark}

    As a consequence of the previous corollary we get the following.

    \begin{proposition}
    \label{prop: x^m has measure zero in free profinite}
        Let $m$ be any positive integer. Then the image of the word $w(x)=x^m$ has measure zero in any non-abelian free profinite group.
    \end{proposition}

    \begin{proof}
     Let $p$ be a prime dividing $m$. Then the word $w(x)=x^m$ is not surjective in $H_{p,d}$ for any $d\geq 2$. Since here the number of variables of $w$ is $k=1$, the claim follows from Corollary \ref{cor: prop A for free profinite}.
    \end{proof}

     \begin{remark}
     The previous result can be obtained also without using Theorem \ref{prop:A}. Namely, by
     \cite{Wa}, the probability that a permutation is an $m$-th power in the symmetric group $S_n$ tends to zero for $n \to \infty$. Thus, since the free profinite group surjects onto any symmetric group, we have, using \eqref{eq: measure r.f. group}, that the measure of the set of $m$-th powers in a free profinite group is zero. 
     \end{remark}

    Now it follows from Proposition \ref{prop: x^m has measure zero in free profinite} that Corollary \ref{cor: in free pro-p groups there is a set with measure 0 that generates open subgroup= verbal set of concatenation} is true also in free profinite groups.

    \begin{corollary}
      Let $d\ge 2$ be an integer and let $\widehat{F}_d$ be the free profinite group of rank $d$. Then there exists a subset $X\subset \widehat{F}_d$ with $X^{-1}=X$, $\mu(X)=0$ and $0<\mu(X^{\ast m})<1$ for some natural number $m$.  
    \end{corollary}

    \begin{proof}
Given a positive integer $n$, take $X=\widehat{F}_d^{\lbrace x^n \rbrace}$. The proof that $X$ has the required properties is the same as the proof of Corollary \ref{cor: in free pro-p groups there is a set with measure 0 that generates open subgroup= verbal set of concatenation}, using that, by work of Nikolov and Segal, the word $x^n$ has finite width in any finitely generated profinite group (\cite[Corollary 2]{NS}).        
\end{proof}

\section{Measure of \texorpdfstring{$\exists$}{E}-definable sets in free pro-\texorpdfstring{$p$}{p} and free profinite groups}
\label{sec: existential}

In this section we will use the notation established in Section~\ref{sec: preliminaries on formulae}. We investigate the $\mathcal{F}$-spectrum of free pro-$p$ and free profinite groups for some special families $\mathcal{F}$ of existential formulae. We already know that there are verbal sets of measure zero and measure one, and also of positive measure different from one. 

\subsection{Zero-one behaviour of some \texorpdfstring{$\mathcal{F}$}{F}-spectra}\label{sec:thmA}
In this section we prove that, if $\mathcal{S}$ is a family of existential formulae $\varphi$ satisfying a certain condition on the number of variables of the words occurring in $\varphi$ relative to the rank $d$ of the free profinite group $\widehat{F}_d$ considered, then $\mathrm{Spec}_{\mu}(\mathcal{S}; \widehat{F}_d)=\lbrace 0,1\rbrace$.

Recall that, given a word $w$, we denote by $\#\mathrm{var}(w)$ the number of variables occurring in $w$. For boolean combinations of words (see Section \ref{sec: preliminaries on formulae}) we immediately obtain the following.

\begin{proposition}
    Let $G$ denote either $H_d$ or $\widehat{F}_d$ with $d\ge 2$ and let $\mathbf{b}$ be a boolean combination of words. Let $\mathcal{B}$ be the family of boolean combinations of words $\mathbf{b}$ such that every word $w$ occurring in $\mathbf{b}$ satisfies one of the following: 
    \begin{itemize}
        \item $w$ is a commutator word, or
        \item $w$ is surjective, or
        \item $w$ is not surjective and satisfies $\# \mathrm{var}(w)+1\leq d$.
    \end{itemize}
     Then $\mathrm{Spec}_\mu(\mathcal{B}; G)=\lbrace 0,1\rbrace$.
\end{proposition}

\begin{proof}
By Theorem \ref{prop:A} the image of each of the words occurring in $\mathbf{b}$ has measure zero or one, and hence the same is true for the sets defined by their negations. Furthermore, note that, by the inclusion-exclusion principle, finite intersections of sets of measure one have measure one. Thus, the set defined by $\mathbf{b}$ has measure 0 or 1, as it is a finite union of finite intersections of sets of measure zero or one. 
\end{proof}

We now turn to existential formulae.
Recall from (\ref{eq: general form of a formula}) that an existential formula is equivalent to a formula of the form

\begin{equation*}
\mathbf{\exists} \boldsymbol{y}:\bigvee_{j=1}^m \left(w_{j, 1}(x,\boldsymbol{y})\#_{j,1} 1\wedge\cdots\wedge w_{j,k_j}(x,\boldsymbol{y})\#_{j,k_j}1\right),
\end{equation*}

where every $w_{j,i}$ is a word, $m$ and the $k_j$'s are non-negative integers, and $\#_{j,i}\in \lbrace =, \neq\rbrace$.

Rearranging the terms if necessary, we can write each term of the disjunction in this formula as

\begin{equation}
\label{eq: term in existential formula}
\mathbf{\exists} \boldsymbol{y}: \left(\bigwedge_{i=1}^\ell w_{i}(x,\boldsymbol{y})= 1\right)\wedge\left(\bigwedge_{i=\ell+1}^k w_{i}(x,\boldsymbol{y})\neq 1\right)
\end{equation}
for a suitable positive integer $k$ and non-negative integer $\ell$.

In what follows, let $w$ be a non-trivial reduced word in $k+1$ variables.
We look first at formulae of the form $\eta_w(x):=\left(\exists\boldsymbol{y}: w(x, \boldsymbol{y})=1\right)$ and prove Theorem \ref{prop: measure spectrum of E-formulae in free pro-p groups}. Recall from Section \ref{sec: preliminaries on formulae} that we call such a formula also equation formula related to~$w$.

\begin{remark}
    Let $G$ be a free pro-$p$ group and let 
    $$S:=\lbrace (P,\boldsymbol{U})\in G\times G^k\mid w(P,\boldsymbol{U})=1\rbrace.$$
    By \cite[Theorem C]{KOTVW}, $\mu(S)=0$. However, this does not yield any information on the measure of the set $\lbrace P\in G\mid\exists \boldsymbol{U}\in G^k : w(P,\boldsymbol{U})=1\rbrace=\pi_G(S)$, where $\pi_G\colon G\times G^k\rightarrow G$ denotes the canonical projection. 
    
    For example, if $w(x, y)=xy^{-1}$, then $\pi_G(S)=G$ and has therefore measure 1, while, if $w(x, y_1,y_2)=x[y_1,y_2]$, then $\mu(\pi_G(S))=0$. 
    
    If $w(x, y_1, y_2, z_1, z_2, u_1, u_2)=x[y_1,z_1]u_1^p[y_2,z_2]u_2^p$, then $\mu(\pi_{H_2}(S))=\frac{1}{p^2}$. 
\end{remark}

The equation $w(x,\boldsymbol{y})=1$ always admits the trivial solution.
Then let $(P, \boldsymbol{U})$ be a solution to the equation in the free pro-$p$ group $H_d$.
Rewrite
\begin{equation}
\label{eq: rewriting w}
w(P,\boldsymbol{U})=P^{-n(w)}W(\boldsymbol{U})\prod_{j=1}^{N(w)}{\Gamma_j(P,\boldsymbol{U})}, 
\end{equation}
for some integers $n(w)$ (namely, the additive inverse of the sum of the exponents of $P$ in $w(P,\boldsymbol{U})$) and $N(w)$, some word $W(\boldsymbol{y})=\prod_{i=1}^k{y_{i}^{n_{i}}}$ and some commutator words $\Gamma_j$. Here, for each $i\in\lbrace 1,\ldots, k\rbrace$, $n_i$ is the sum of the exponents of the variable $y_i$ in $w$. 

\begin{remark}
\label{rmk: algorithm to find W}
Note that, in the previous equality (\ref{eq: rewriting w}), $n(w)$ and $W$ are uniquely determined from $w$. Moreover, setting $\Gamma:=\prod_j{\Gamma_j}$, $\#\mathrm{var}(W)=\#\mathrm{var}(w)-1=k$ and $\#\mathrm{var}(\Gamma)=\#\mathrm{var}(w)=k+1$. 

In order to obtain the expression (\ref{eq: rewriting w}) starting from $w$ without ambiguities, we will use the following algorithm: 
we move first all occurrences of $P$ to the left, one after the other, until we obtain $P^{-n(w)}$. We then move each $U_i$ to the left, next to $P^{-n(w)}$, without changing their order, until we obtain an expression of the form $P^{-n(w)}\prod_{m}{U_{i_m}^{l_m}}\prod_j\widetilde{\Gamma}_j(P,\boldsymbol{U})$, where $\widetilde{\Gamma}_j$ are commutator words.
Finally, we collect all instances of the $U_i$'s in order to obtain $W$. 

For instance, if $w(P,\boldsymbol{U})=P^3U_2^{3}P^{-1}U_1^5$, our algorithm has the following steps:

\begin{align*}
     w(P,\boldsymbol{U})&=P^2U_2^{3}[U_2^3,P^{-1}]U_1^5\\
    & =P^2U_2^{3}U_1^5[U_2^3,P^{-1}][[U_2^3,P^{-1}],U_1^5]\\
    & =P^2U_1^5U_2^{3}[U_2^3,U_1^5][U_2^3,P^{-1}][[U_2^3,P^{-1}],U_1^5].
\end{align*}

In this case, we find $W(y_1,y_2)=y_1^5y_2^3$, and 
$$ \Gamma_1(x, y_1,y_2)=[y_2^3,y_1^{5}],\
  \Gamma_2(x, y_1,y_2)=[y_2^3,x^{-1}],\
  \Gamma_3(x, y_1,y_2)=[[y_2^3,x^{-1}],y_1^5].$$
\end{remark}

Continuing our analysis of the equation $w(x,\boldsymbol{y})=1$, using \eqref{eq: rewriting w} we can write the equality $w(P,\boldsymbol{U})=1$ as $$P^{n(w)}=W(\boldsymbol{U})\prod_{j=1}^{N(w)}\Gamma_j(P,\boldsymbol{U})=\prod_{i=1}^k{U_{i}^{n_{i}}}\prod_{j=1}^{N(w)}\Gamma_j(P,\boldsymbol{U}).$$

Let $l(w)$ be the greatest common divisor of the exponents $n_i$ of the variables $U_i$ in the word $W$, or $l(w)=0$ if $W$ is the trivial word.

If $n(w)=0$, then any $P\in H_d$ together with the tuple $\boldsymbol{U}$ with each $U_i=1$ gives a solution to the equation $w(P,\boldsymbol{U})=1$. Indeed, in this case $W(\boldsymbol{U})\equiv 1$ and each commutator is trivial. It follows that, if $n(w)=0$, then the measure of $G_{\eta_w}$ is $1$. 

If instead $n(w)\neq 0$, we distinguish two cases: 
\begin{itemize}
    \item $l(w)$ divides $n(w)$;  
    \item $l(w)$ does not divide $n(w)$. 
\end{itemize}

If $l(w)$ divides $n(w)$, any $P$ is a solution of $\eta_w$ by taking as $U_i$ some appropriate powers of $P$.

Otherwise, if $l(w)$ does not divide $n(w)$, we distinguish two subcases:
\begin{enumerate}[leftmargin=7.5mm]
    \item[(a)] \emph{$l(w)=0$ (that is, $W$ is trivial)}.
Then $P^{n(w)}$, and hence $P$, belong to the commutator subgroup of $H_d$, which has infinite index in $H_d$. So in this case $G_{\eta_w}$ has measure $0$.     
    \item[(b)] \emph{$l(w)\neq 0$, $p\mid \gcd(1-n(w),l(w))$ and $\#\mathrm{var}(w)\leq d-1$}. 
If $W$ is not trivial,
define the word in $\#\mathrm{var}(w)$ variables 
\begin{equation}
\label{eq: w tilde}
\tilde{w}(x,\boldsymbol{y}):=x^{-(n(w)-1)}W(\boldsymbol{y})\prod_{j=1}^{N(w)}\Gamma_j(x,\boldsymbol{y})
\end{equation}
obtained from $w$ as described above. As we will see in Lemma \ref{lem: surjectivity conditions for words in free pro-p},
$\tilde{w}$ is surjective if and only if $p\nmid \gcd(1-n(w),l(w))$.

If $P$ is a solution to $\eta_w$, then $P$ lies in the image of $\tilde{w}$.
If $p\mid \gcd(1-n(w),l(w))$, then $\tilde{w}$ is not surjective on $H_d$. Hence, if $\# \mathrm{var}(\tilde{w})=\# \mathrm{var}(w)\leq d-1$, then, by Theorem \ref{prop:A}, the set $(H_d)_{\eta_w}$ is contained in a set of measure $0$. 
It follows that the set of solutions to $\eta_w$ has also measure zero.
\end{enumerate}

We summarise the conditions that we have found on the word $w$ in the following definition.

\begin{definition}
    \label{def: 0-1 word}
    Let $d\ge 2$ be an integer.
    Let $w(x,\boldsymbol{y})$ be a word in $\#\mathrm{var}(w)=k+1$ variables and let $-n(w)$ be the sum of the exponents of $x$ in $w$. 
    Let $W(\boldsymbol{y}):=\prod_{i=1}^k{y_i^{n_i}}$, where each $n_i$ is the sum of the exponents of $y_i$ in the word $w$, and let $l(w)$ be the greatest common divisor of the $n_i$'s. 

    We say that $w$ is a \emph{$(p,d)$-zero-one word} if it satisfies one of the following conditions: 
    \begin{enumerate}
    \item $n(w)=0$; 
    \item $n(w)\neq 0$ and $l(w)\mid n(w)$;
    \item $n(w)\neq 0$ and $l(w)=0$;
    \item $n(w)\neq 0$, $l(w)\neq 0$, $l(w)\nmid n(w)$, $p\mid \gcd(1-n(w),l(w))$  
    and $\#\mathrm{var}(w)\leq d-1$.
    \end{enumerate}

    We say that $w$ is a \emph{$d$-zero-one word} if it is a $(p,d)$-zero-one word for some prime $p$.
\end{definition}

\begin{remark}
    Clearly a word $w$ is a $d$-zero-one word if and only if it satisfies either one of conditions (1), (2), (3) in Definition \ref{def: 0-1 word}, or the condition
    $$n(w)\neq 0,\ l(w)\neq 0,\ l(w)\nmid n(w),\ \gcd(1-n(w),l(w))\neq 1\ \text{and}\ \#\mathrm{var}(w)\leq d-1.$$
\end{remark}    

Theorem~\ref{thm:0-1_intro} now follows immediately from the above discussion. 
Moreover, the proof of Corollary \ref{cor: measure spectrum of E-formulae in free profinite groups} for the free profinite group $\widehat{F}_d$ is identical to the pro-$p$ case except for the case $n(w)\neq 0$ and $l(w)$ does not divide $n(w)$. Here it is enough to observe that $(\widehat{F}_d)_{\eta_w}$ projects onto $(H_{p,d})_{\eta_w}$ for every $p$, and has therefore measure zero.

Note that, if $w$ satisfies conditions (1) or (2) in Definition \ref{def: 0-1 word}, then the measure of the set defined by $\eta_w$ is 1, while, if $w$ satisfies conditions (3) or (4), then such measure is 0.

From the proof of Corollary \ref{cor: prop A for free profinite} we get: 

\begin{corollary}
 Let $\widehat{F}_d$ be the free profinite group of rank $d\ge 2$. 
    Consider the family $\mathcal{S}$ containing formulae $\eta_w(x)$ defined by $\exists\boldsymbol{y}: w(x,\boldsymbol{y})=1$, where the word $w$ runs over the set of words with $\gcd(1-n(w),l(w))\neq 1$ and $\#\mathrm{var}(w)\leq d-1$.
    Then $\mathrm{Spec}_\mu(\mathcal{S}; \widehat{F}_d)=\lbrace 0,1\rbrace$.  
\end{corollary}

Next, we consider the measure of the sets defined by inequalities $\iota_w(x):=\exists\boldsymbol{y}: w(x,\boldsymbol{y})\neq 1$.

\begin{lemma}
    \label{lem: measure sets defined by inequalities}
    Let $G$ be a free pro-$p$ (respectively~free profinite) group, and let $\mathcal{I}$ be the collection of formulae $\iota_w(x):= \exists\boldsymbol{y}: w(x,\boldsymbol{y})\neq 1$, where $w$ is any non-trivial word. Then $\mathrm{Spec}_\mu(\mathcal{I}; G)=\lbrace 1\rbrace$.
\end{lemma}

\begin{proof}
Let $w$ be any non-trivial word in $k+1$ variables.
We want to show that 
  $$\mu((G)_{\iota_w})=\lbrace P\in G\mid \exists \boldsymbol{U}\in G^k: w(P,\boldsymbol{U})\neq 1\rbrace=1.$$ 
 
Let $S:=\lbrace (P, \boldsymbol{U})\in G\times G^k: w(P,\boldsymbol{U})= 1\rbrace$. Then, if $G$ is free pro-$p$, we know by \cite[Theorem C]{KOTVW} that $\mu(S)=0$.
On the other hand, if $G$ is free profinite, it is enough to consider the projection of $S$ onto a free pro-$p$ quotient and use Corollary \ref{cor: measure grows with quotient} to deduce that $\mu(S)=0$ also in this case.

In both cases it follows that the measure of the complement $C$ of $S$ is $$\mu(C)=\mu(\lbrace (P, \boldsymbol{U})\in G\times G^k: w(P,\boldsymbol{U})\neq 1\rbrace)=1.$$

 Let $\pi:=\pi_G\colon G\times G^k\rightarrow G$ be the projection on the first coordinate. Then $(G)_{\iota_w}=\pi(C)$ is open, and therefore measurable. Since $C\subseteq \pi^{-1}(\pi(C))$, it follows from Proposition \ref{prop: quotient measure of a set} that $$1=\mu_{G\times G^k}(C)\leq \mu_{G\times G^k}(\pi^{-1}(\pi(C)))=\mu_{G}(\pi(C)).$$ We conclude that $\mu_G(G_{\iota_w})=\mu_G(\pi(C))=1$.
 \end{proof}

We conclude this section with some remarks on general formulae. Using the inclusion-exclusion principle, we immediately obtain the following.

\begin{corollary}
If $w$ is a $(p,d)$-zero-one word (respectively, a $d$-zero-one word) (see Definition \ref{def: 0-1 word}), then boolean combinations of formulae of the form $\eta_w(x)$ and $\iota_w(x)$ with disjoint bound variables define sets of either measure 0 or 1 in the free pro-$p$ groups (respectively, in free profinite groups) of rank $d$. 
\end{corollary}

We now give some examples for the measures of sets defined by positive existential formulae. In particular, we give some conditions that imply that the set defined by such a formula has measure zero or one. Since a word formula is a special case of positive existential formula, we already know that, in general, it is possible to find sets defined by such formulae that have measure zero, or one, or positive but different from one.

A \emph{positive existential formula} $\varphi$ consists of the conjunction of finite systems of word equations. More precisely, $\varphi$ is equivalent to a formula of the form

\begin{equation}
\label{eq: positive existential}
\mathbf{\exists} \boldsymbol{y}:\bigvee_{j=1}^m \bigwedge_{i=1}^\ell w_{j,i}(x,\boldsymbol{y})= 1,
\end{equation}

for some positive integers $m$ and $\ell$, and some words $w_{j,i}$. 

\begin{lemma}
\label{lem: sets defined by positive existential formulae are measurable}
    Let $\varphi$ be a positive existential formula and let $G$ be a profinite group. Then $G_\varphi$ is closed, hence measurable.
\end{lemma}

\begin{proof}
    For every $j\in\lbrace1,\ldots,m\rbrace$, let $\varphi_j$ be the formula $\mathbf{\exists} \boldsymbol{y}: \bigwedge_{i=1}^\ell w_{j,i}(x,\boldsymbol{y})= 1$. Then $G_{\varphi_j}=\pi_G(\lbrace(x,\boldsymbol{y})\mid w_{j,1}(x,\boldsymbol{y})=\cdots=w_{j,\ell}(x,\boldsymbol{y})=1\rbrace)$ is a closed set of $G$. Since $G_\varphi=\cup_{j=1}^mG_{\varphi_j}$ is a finite union of closed sets, it is closed, hence measurable.
\end{proof}

\begin{example}
\label{ex: measure set defined by positive existential}
Let $\varphi$ be a positive existential formula and let $G$ denote either $H_d$ or $\widehat{F}_d$. In this example we will find some conditions on $\varphi$ under which $\mathrm{Spec}_\mu(G,\varphi)=\lbrace 0,1\rbrace$.

Note that, by using the inclusion-exclusion principle, it is enough to consider the case when $\varphi$ is equivalent to $\mathbf{\exists} \boldsymbol{y}:\bigwedge_{i=1}^\ell w_{i}(x,\boldsymbol{y})= 1$.
Then $G_\varphi\subseteq \cap_{i=1}^\ell G_{ \eta_{w_i}}$.
It follows that, if one of the words $w_i$ occurring in $\varphi$ is such that $\mu(G_{\eta_{w_i}})=0$ for some $i\in\lbrace 1,\ldots, \ell\rbrace$, then $\mu(G_\varphi)=0$. 

Then suppose that all $w_i$ satisfy $\mu(G_{\eta_{w_i}})=1$. 
Using the previous notation, for each $i$ denote by $n(w_i)$ the sum of the exponents of $x$ occurring in $w_i(x,\boldsymbol{y})$ and by $l_i$ the greatest common divisor of the sum of the exponents of the $y_j$'s in $w(x,\boldsymbol{y})$. 
Rewrite $\varphi$ as
$$\mathbf{\exists} \boldsymbol{y}:\bigwedge_{i=1}^{\ell'} w_{i}(x,\boldsymbol{y})= 1\wedge \bigwedge_{i=\ell'+1}^{\ell} w_{i}(x,\boldsymbol{y})= 1,$$ where $n(w_i)\neq 0$ for all $i\in \lbrace 1,\ldots, \ell'\rbrace$ and $n(w_i)=0$ for all $i\in\lbrace \ell'+1,\ldots,\ell\rbrace$, and let $k$ be the length of the string $\boldsymbol{y}$.

If $\ell'=0$, then every element of $G$ is a solution of $\varphi$ and $\mu(G_\varphi)=1$. Indeed, in this case, given any $P\in G$, substituting $U_j=1$ for every $y_j$ yields a solution to the formula in $G$. 

Hence assume that $\ell'\neq 0$. If the sets of $y_j$'s occurring in the equations are pairwise disjoint, then $\mu(G_\varphi)=1$ by the inclusion-exclusion principle.

We now find some conditions that imply $\mu(G_\varphi)=0$ under the assumption $\ell'\ge 2$. 
Denote by $n_{i,j}$ the sum of the exponents of $y_j$ in $w_i$.

 Suppose that $\ell'\ge 2$ and take any $i_1\neq i_2\in\lbrace 1,\ldots,\ell'\rbrace$. Consider the sum of the exponents $n_{i_1+i_2,j}:=n_{i_1,j}+n_{i_2,j}$ of $y_j$ occurring in the equations $w_{i_1}(x,\boldsymbol{y})=1$ and $w_{i_2}(x,\boldsymbol{y})=1$. Moreover, assume that $n(w_{i_1})+n(w_{i_2})\neq 0$ for at least one pair $(i_1,i_2)$. 
 \begin{itemize}[leftmargin=4mm]
 \item[(a)] If $n_{i_1+i_2,j}=0$ for every $j$, then, multiplying the $i_1$-th and the $i_2$-th equation we obtain:
 $$x^{n(w_{i_1})+n(w_{i_2})}=\tilde{\Gamma}(x,\boldsymbol{y)}\prod_{j_{1}=1}^{N(w_{i_1})}{\Gamma_{j_{1}}(x,\boldsymbol{y})}\prod_{j_{2}=1}^{N(w_{i_2})}{\Gamma_{j_{2}}(x,\boldsymbol{y})},$$
 for some product of commutators $\tilde{\Gamma}(x,\boldsymbol{y})$.
 Therefore, any solution to $\varphi$ must lie in the commutator subgroup of $G$ and therefore $\mu(G_\varphi)=0$.
 
 \item[(b)] 
  Now suppose that $n_{i_1+i_2,j}\neq 0$ for at least one $j$ and that $\mathrm{gcd}_j(n_{i_1+i_2,j})$ does not divide $n(w_{i_1})+n(w_{i_2})$.
 Multiplying also in this case the $i_1$-th and $i_2$-th equations, we find an equality of the form $x^{n(w_{i_1})+n(w_{i_2})}=\tilde{w}(x,\boldsymbol{y})$. If this equation satisfies condition (4) of Definition \ref{def: 0-1 word}, then, any solution to the original formula is a solution to a formula whose solution set has measure 0, and so $\mu(G_\varphi)=0$.
 \end{itemize}
 \end{example}

\begin{remark}
Note that the cases presented in the previous example are far from exhaustive. 
In particular, in this remark we give three examples corresponding to three different behaviours not represented in the previous Example \ref{ex: measure set defined by positive existential}. All three examples concern a formula $\varphi$ of the form
$$\exists\boldsymbol{y}: w_1(x,\boldsymbol{y})=1\wedge w_2(x,\boldsymbol{y})=1$$
where $\mu(G_{\eta_{w_i}})=1$ for $i\in\lbrace 1,2\rbrace$.
We will find: in (1) $\mu(G_\varphi)=0$, in (2)  $0<\mu(G_\varphi)<1$ and, in (3), $\mu(G_\varphi)=1$, showing that all behaviours are possible.

\begin{enumerate}

\item Consider the formula $\exists y_1,y_2:x=y_1y_2^3=y_1^2y_2^3$. Then every element in $G$ is solution to each equation considered separately. However, if $P$ is a solution to the formula (say with $y_i$ replaced by $U_i\in G$ for every $i\in\lbrace 1,2\rbrace$), then necessarily $U_1=1$, and therefore $P$ is a cube. Hence $G_\varphi\subseteq G^{\lbrace x^3\rbrace}$ has measure 0. Here, the difference with the second case of (2) in Example \ref{ex: measure set defined by positive existential} is that $\#\mathrm{var}(\tilde{w})=2$ but $\mu(G_\varphi)=0$ also holds when $d=2$ (i.e., here the condition $\#\mathrm{var}(\tilde{w})\le d-1$ is not necessary).

\item As an example of two sets of measure one whose intersection has positive measure different from one, consider $G=H_2$ the free pro-$p$ group of rank $2$ with $p\neq 2$, and the formula $\varphi$ given by $$\exists y_1,\ldots,y_6: x=y_1y_2^p[y_3,y_4][y_5,y_6]=y_1^2y_2^p[y_3,y_4][y_5,y_6].$$ Then each component of $\varphi$ defines a set of measure one. However, $G_\varphi\subseteq\lbrace P: \exists U_2,\ldots,U_6: P=U_2^p[U_3,U_4][U_5,U_6]\rbrace:=S$ and therefore $\mu(G_\varphi)\leq \frac{1}{p^2}$. Since the set $S$ is clearly contained in $G_\varphi$, we can conclude that $\mu(G_\varphi)=\frac{1}{p^2}$. 
Note that, if $d>5$, we would find $\mu((H_d)_\varphi)=0$ (compare with Example \ref{ex: measure set defined by positive existential} for $\ell'\ge 2$, case (b) with $p=3$).

\item Finally, the formula $\exists y_1,y_2: x=y_1y_2=y_1^2y_2$ defines a set of measure one. This example is not included in Example \ref{ex: measure set defined by positive existential} because here $\ell'\neq 0$ and the bounded variables in the formulae $\eta_{w_1}$ and $\eta_{w_2}$ are not disjoint. 
\end{enumerate}
\end{remark}

In  conclusion, the following question remains open, in general. 

\begin{question}
Using the same notation as in Example \ref{ex: measure set defined by positive existential}, what is the relation between the words $w_i$ and $\mu(G_\varphi)$ when $\mu(G_{\eta_{w_i}})=1$ for every $i\in\lbrace 1,\ldots, \ell\rbrace$, $\ell'\neq 0$ and the sets of bounded variables occurring in the equations are not pairwise disjoint?
\end{question}

\subsection{Abstract free groups vs. profinite free groups}

In this section we make some remarks regarding the measure of definable sets of abstract free groups, and compare what happens in free abstract groups and in free profinite groups.

Let $S$ be a subset of a residually finite group $G$.
Since the closure $\overline{S}$ of $S$ in the profinite completion $\widehat{G}$ is measurable, then we say that the Haar measure $\mu_\infty(S)$ of $S$ is the measure of its closure in $\widehat{G}$.

Analogously, if $G$ is residually-finite-$p$ group and $S$ is a subset whose closure is measurable in the pro-$p$ completion of $G$, we say that the Haar measure $\mu_p(S)$ of $S$ is the measure of its closure.

By slight abuse of notation, we will denote by $\mathrm{Spec}_{\mu_\infty}(\mathcal{F}, F_d)$ the possible measures of \emph{the closure} of $(F_d)_\varphi$ in $\widehat{F}_d$, for $\varphi\in\mathcal{F}$. Similar notation will be used for $\mathrm{Spec}_{\mu_p}(\mathcal{F}, F_d)$

\begin{remark}\label{rmk:free_abs}
Let $w$ be a $(p,d)$-zero-one word (respectively, a $d$-zero-one word).
Theorem \ref{prop: measure spectrum of E-formulae in free pro-p groups} and Corollary \ref{cor: measure spectrum of E-formulae in free profinite groups} concern the measure of the sets of solutions of equations $\eta_w$ in free pro-$p$ and free profinite groups. Since $G_{\eta_w}$ is closed, the measure of $G_{\eta_w}$ is one exactly when every element of the group is a solution of the formula. Using the previous notation, this happens when $n(w)=0$ and when $n(w)\neq 0$ and $l(w)$ divides $n(w)$. In particular, if we look at the same formula in a discrete free group $F_d$ (or, actually, in any group), using the same argument, under these conditions every element of $F_d$ is a solution of $\eta_w$. Hence, if we consider the measures $\mu_\infty$ and $\mu_p$, under these conditions on $w$ we obtain that $\mu_\infty((F_d)_{\eta_w})=\mu_p((F_d)_{\eta_w})=1$.  
    
Analogously, assume that $\mu_{H_d}((H_d)_{\eta_w})=0$ (respectively, $\mu_{\widehat{F}_d}((\widehat{F}_d)_{\eta_w})=0$).
Since $\eta_w$ is an existential formula, $(F_d)_{\eta_w}$ embeds into $(H_d)_{\eta_w}$ and into $(\widehat{F}_d)_{\eta_w}$. The latter sets being closed, the closure of $(F_d)_{\eta_w}$ in the respective groups embeds into $(H_d)_{\eta_w}$ and $(\widehat{F}_d)_{\eta_w}$ as well. 
Then the closure of $(F_d)_{\eta_w}$ embeds into a set of measure 0 in $H_d$ (respectively $\widehat{F}_d$) and, as a consequence, it has measure 0 as well. It follows that, if we restrict to the family $\mathcal{E}(p)$ (respectively, $\mathcal{E}$) of equations $\eta_w$ associated to $(p,d)$-zero-one words (respectively, $d$-zero-one words), then $\mathrm{Spec}_{\mu_\infty}(\mathcal{E}; F_d)=\mathrm{Spec}_{\mu_p}(\mathcal{E}(p); F_d)=\lbrace 0,1\rbrace$.
\end{remark}

Merzljakov proved that (abstract) free groups have \emph{trivial} positive theory (\cite{Me}). This means that a positive sentence that is true in any non-abelian free group is true in every group. We will use Remark~\ref{rmk:free_abs} to show that something analogous happens for free profinite groups and equation formulae, namely that an equation formula that is satisfied by every element of some non-abelian free profinite group is satisfied by every element of every group. In order to do so, we need the following lemma.

\begin{lemma}
\label{lem: n neq 0 and l does not divide n; measure is not 1}
    Let $\widehat{F}_d$ be the free profinite group of rank $d\ge 2$.
    Using the previous notation, assume that the word $w$ satisfies the following conditions: 
    \begin{enumerate}
    \item $n(w)\neq 0$, and 
    \item $l(w)\neq 0$, and 
    \item $l(w)\nmid n(w)$. 
    \end{enumerate}
    
    Then $\mu((\widehat{F}_d)_{\eta_w})\neq 1$.

    If $F_d$ denotes the abstract free group of rank $d$, then, under the same hypotheses on $w$, not every element of $F_d$ is a solution to $\eta_w$ in $F_d$. 
\end{lemma}

\begin{proof}
Let $G$ denote $\widehat{F}_d$.
Recall that, by Lemma \ref{lem: measure closed set =1 iff C=G}, if $C$ is any closed set in $G$, then $\mu(C)=1$ if and only if  $C=G$. 

Consider a word $w$ with $k+1$ variables such that $n(w)\neq 0$ and $l(w)$ does not divide $n(w)$. Let $P\neq 1$ be a solution in $G$ to the formula $\exists\boldsymbol{y}: w(x,\boldsymbol{y})=~1$.
In the abelianisation $G_{\text{ab}}$ of $G$ we have $\overline{P}^{n(w)}=\overline{U}_1^{n_1}\cdots \overline{U}_k^{n_k}$, and $l(w)=\mathrm{gcd}_i(n_i)$. Then $\overline{P}^{n(w)}\in G_{\text{ab}}^{l(w)}$. Since $G_{\text{ab}}$ is torsion-free, extraction of roots is unique, so we can reduce to the case $\gcd(n(w),l(w))=1$. If every element of $G$ was a solution to $\eta_w$, then every element of $G_{\text{ab}}$ would be a solution to $\eta_w$, from which one would obtain that every $n(w)$-th power in $G_{\text{ab}}$ is an $l(w)$-th power. Since this is impossible (for instance, take as $P$ a generator of $G_{\text{ab}}$), we conclude that $\mu(G_{\eta_w})<1$.

The same argument yields that, under the given hypotheses on $w$, not every element of $F_d$ is a solution to $\eta_w$.
\end{proof}

\begin{corollary}
\label{cor: trivial eq in free profinite groups}
    Let $\widehat{F}_d$ be the free profinite group of rank $d\ge 2$ and let $\eta_w$ be an equation formula related to the word $w$. 
    Then the sentence $\forall x:\eta_w(x)$ is true in $\widehat{F}_d$ if and only if it is true in every group.
\end{corollary}

\begin{proof}
Write $G=\widehat{F}_d$. Combining the proof of Theorem~\ref{prop: measure spectrum of E-formulae in free pro-p groups} and Lemma~\ref{lem: n neq 0 and l does not divide n; measure is not 1} we get that the sentence $\forall x:\eta_w(x)$ is true in $G$ (or, equivalently, $\mu(G_{\eta_w})=1$) if and only if $n(w)=0$ or $n(w)\neq 0$ and $l(w)$ divides $n(w)$. In these cases, every element of any group is a solution to $\eta_w$, which proves the claim.
\end{proof}

From Lemma \ref{lem: n neq 0 and l does not divide n; measure is not 1} we can also easily deduce when a word is surjective in every free profinite group of finite rank.

\begin{corollary}
    \label{cor: surjective words}
Let $G:=\widehat{F}_d$ with $d\ge 2$ and let $\eta_w$ be an equation formula related to the word $w$. With the notation introduced in this section, $\mu((G)_{\eta_w})=1$ if and only if one of the following conditions is satisfied: 
\begin{enumerate}
    \item $n(w)=0$, or 
    \item $n(w)\neq 0$ and $l(w)\mid n(w)$.
\end{enumerate}

In particular, if we consider a word formula of the form $$\exists\boldsymbol{y}: x=\widehat{w}(\boldsymbol{y})=:W(\boldsymbol{y})\prod_{j}\Gamma_j(\boldsymbol{y}),$$ where $\widehat{w}$ is a word and $W$ and $\Gamma_j$ are a suitable word and product of commutators obtained from $\widehat{w}$, then $\mu(G^{\lbrace \widehat{w}+\rbrace})=1$ if and only if the greatest common divisor $l(\widehat{w})$ of the exponents of the $y_i$ in $W(\boldsymbol{y})$ is equal to one. 
\end{corollary}

\begin{proof}
  The first part of the corollary follows immediately from the proof of Theorem \ref{prop: measure spectrum of E-formulae in free pro-p groups} and from Lemma \ref{lem: n neq 0 and l does not divide n; measure is not 1}.

  For the second claim, apply the first part of the corollary to the formula $\exists\boldsymbol{y}: x^{-1}\widehat{w}(\boldsymbol{y})=1$.
\end{proof}

In the pro-$p$ case there is one further condition that implies that a word is surjective.

\begin{lemma}
\label{lem: surjectivity conditions for words in free pro-p}
Let $G_p:=H_{p,d}$ for $d\ge 2$.

\begin{enumerate}
\item Let $\eta_w$ be an equation formula related to the word $w$. With the notation introduced in this section, if $\mu((G_p)_{\eta_w})=1$, then one of the following conditions is satisfied: 

\begin{enumerate}
\item $n(w)=0$, or 
\item $n(w)\neq 0$ and $l(w)\mid n(w)$, or
\item $n(w)\neq 0$, $l(w)\nmid n(w)$  and $p\nmid \gcd(1-n(w),l(w))$.
\end{enumerate}

\item If we consider a formula $\varphi$ of the form $$\exists\boldsymbol{y}: x^{n(\widehat{w})}=\widehat{w}(\boldsymbol{y})=:W(\boldsymbol{y})\prod_{j}\Gamma_j(\boldsymbol{y}),$$ then $\mu((G_p)_\varphi)=1$ if, denoting by $l(\widehat{w})$ the greatest common divisor of the exponents of the $y_i$ in $W(\boldsymbol{y})$, one of the following conditions is satisfied:

\begin{enumerate}
\item[(d)] $n(\widehat{w})=0$, or 
\item[(e)] $n(\widehat{w})\neq 0$ and $l(\widehat{w})\mid n(\widehat{w})$, or
\item[(f)] $n(\widehat{w})\neq 0$, $l(\widehat{w})\nmid n(\widehat{w})$  and $p\nmid l(\widehat{w})$.
\end{enumerate}

\item In particular, if $n(\widehat{w})=1$, we obtain that $\mu(G_p^{\lbrace\widehat{w}+\rbrace})=1$ if and only if $p\nmid l(\widehat{w})$. 
\end{enumerate}
\end{lemma}

\begin{proof}
\quad
   
 \begin{enumerate}[leftmargin=7mm]

 \item Assume that $\mu((G_p)_{\eta_w})=1$, and suppose that none of the three given conditions is satisfied. The only remaining possibility (excluding the case $W\equiv 1$, or, equivalently, $l(w)=0$, which leads to measure zero), is $n(w)\neq 0$, $l(w)\nmid n(w)$ and $p\mid \gcd(1-n(w),l(w))$. 

 Recall that, if $P$ is a solution to $\eta_w$, then $$P^{n(w)}=W(\boldsymbol{U})\prod_{j=1}^{N(w)}{\Gamma_j(P,\boldsymbol{U})},$$ from which we obtain
 $$P=\tilde{w}(P,\boldsymbol{U})=P^{1-n(w)}W(\boldsymbol{U})\prod_{j=1}^{N(w)}{\Gamma_j(P,\boldsymbol{U})},$$
 where we use the same notation as in \eqref{eq: w tilde}.

 Since $p\mid \gcd(1-n(w),l(w))$, the projection of the verbal set $G_p^{\lbrace\tilde{w}+\rbrace}$ onto the abelianisation of $G_p$ does not have measure one, and therefore the measure of $G_p^{\lbrace\tilde{w}+\rbrace}$ is not one. Because $(G_p)_{\eta_w}\subseteq G_p^{\lbrace\tilde{w}+\rbrace}$, it follows that $\mu((G_p)_{\eta_w})\neq~1$.
 
 \item 
We already know that conditions $(d)$ and $(e)$ imply $\mu((G_p)_{\varphi})=1$. 
  
 Hence, suppose that condition $(f)$ is satisfied. By definition of $l(\widehat{w})$, choosing appropriate multiples $\alpha_1,\ldots,\alpha_k$ of the exponents $n_1,\ldots, n_k$ of the bounded variables $y_1,\ldots,y_k$, one obtains $l(\widehat{w})=\alpha_1n_1+\cdots+\alpha_kn_k$. Since $p\nmid l(\widehat{w})$ the map $z\mapsto z^{l(\widehat{w})}$ is surjective on $G_{p}$. Therefore, for every $P\in G_p$, there exists $Q\in G_p$ such that, setting $U_i:=Q^{\alpha_i}$ for every $i\in\lbrace 1,\ldots, k\rbrace$, one obtains that $P^{n(w)}=\widehat{w}(\boldsymbol{U})=W(\boldsymbol{U})\prod_{j}{\Gamma_j(\boldsymbol{U})}=W(\boldsymbol{U})=Q^{l(\widehat{w})}$. It follows that every $P\in G_p$ satisfies $\varphi$, which yields the claim.

 \item The last statement follows directly combining (1) and (2) with $n(w)=n(\widehat{w})=1$.
\end{enumerate}
\end{proof}

We conclude with an observation comparing the Haar measure of definable sets in free profinite groups with a different way of measuring the size of definable sets in abstract free groups.

\begin{remark}
The situation for profinite groups differs from the classical situation studied for discrete free groups. Recall the following definition (\cite[Definition 19]{KM}): let $F$ be a (discrete) free group with generating set $X$. A subset $S\subseteq F$ is \emph{generic} if 
    $$\lim_{n\rightarrow \infty}\frac{\vert S\cap B_n(X)\vert}{\vert B_n(X)\vert}=1,$$
where $B_n(X)$ is the ball of radius $n$ in the Cayley graph of $F$. A set $S\subseteq F$ is \emph{negligible} if its complement is generic.

By \cite[Corollary 18]{KM}, a definable set in a free group is either negligible or generic. Therefore, in this setting definable sets have density zero or one in abstract free groups, while we have seen that in free profinite groups there are definable sets of positive Haar measure different from zero and one. 
\end{remark}

\section{Measure of definable sets in abelian groups}

\subsection{Gap in the measure of definable sets of abelian groups}

Let $\mathcal{PI}$ be the family of identitites of the form $w(\boldsymbol{g})=1$, where $w$ is a twisted word map in $d$ variables (see \cite[Section 2]{M} for the definition). Then, according to \cite[Theorem 3.3]{M}, if $G$ is a nilpotent profinite group of class $k$, $\mathrm{Spec}_\mu(\mathcal{PI}, G^d)\subseteq [0, \frac{2^k-1}{2^k}]\cup\lbrace 1\rbrace$. 
This is an example of \emph{gap} in the $\mathcal{PI}$- spectrum of $G^d$: if $\mu_{G^d}(\lbrace \boldsymbol{g}\in G^d: w(\boldsymbol{g})=1\rbrace)>\frac{2^k-1}{2^k}$, then every tuple $\boldsymbol{g}$ of $G^d$ satisfies $w(\boldsymbol{g})=1$ (and hence $w$ is an identity in $G$). In the next example we see that, in general, this gap does not hold true for formulae. We look at abelian groups, where $\frac{2^k-1}{2^k}=\frac{1}{2}$.

\begin{example}
    Let $C_6$ be the finite cyclic group of order $6$, and, using additive notation, consider the formula $\varphi(x):= (2x=0)\vee (3x=0)$. Then $\vert (C_6)_\varphi\vert = 4$ and $\mu((C_6)_\varphi)=\frac{\vert (C_6)_\varphi\vert}{\vert C_6\vert}=\frac{2}{3}>\frac{1}{2}$, but not every element of $C_6$ satisfies $\varphi$. 
\end{example}

More generally, we imitate the previous example to prove that no gap can exist when considering the class of all formulae (with one free variable), evaluated in abelian groups.

\begin{proposition}
    Let $\mathscr{A}$ be the class of abelian profinite groups. Then there is no real number $\gamma\in [0,1]$ such that, for every formula $\varphi$ and every $G\in\mathscr{A}$, $\mu(G_\varphi)>\gamma$ implies $\mu(G_\varphi)=1$. 
\end{proposition}

\begin{proof}
    For any natural number $n$ let $C_n$ be the cyclic group of order $n$. Let $p$ and $q$ be two distinct primes and consider the group $G(p,q):=C_p\times C_q$ of order $n:=pq$. Let $\varphi_{p,q}$ be the formula $(px=0)\vee (qx=0)$. Denoting by $\phi$ the Euler's function, the elements satisfying this formula are exactly the non-generators of $G(p,q)$, which are $n-\phi(n)$. It follows that $\mu(G(p,q)_{\varphi_{p,q}})=\frac{n-\phi(n)}{n}=1-\frac{\phi(n)}{n}$. We conclude that $\limsup_{p,q\to\infty}{\mu(G(p,q)_{\varphi_{p,q}})}=1$, which implies that no gap exists.
\end{proof}

\subsection{\texorpdfstring{$\mathcal{F}$}{F}-spectrum in finitely generated abelian groups}
\label{sec: measures in fg abelian groups}
We use additive notation.
By quantifier elimination, it is known that definable subsets of abelian groups are boolean combinations of the sets defined by the formulae $\delta_n:=\exists y: x=ny$ and $\omega_m:=mx=0$, for integers $n$ and $m$ (\cite[Theorem A.2.2]{Hod93}). If $G$ is a finitely generated abelian group, then $G\cong\Z^r\oplus T$, where $r$ is the rank of $G$ and $T$ its torsion subgroup. Clearly $G_{\delta_n}=nG$ and $G_{\omega_m}=\lbrace g\in G\mid mg=0\rbrace$ are subgroups of $G$. It follows that the measures of these definable sets correspond to their indices in $G$.

For simplicity, assume that $G$ is finitely generated of rank $r$ and torsion-free. Given a formula $\varphi$, we will give an explicit expression for $\mu(G_\varphi)$. 

Note that $\mu(G_{\delta_n})=\frac{1}{n^r}$, $\mu(G_{\neg\delta_n})=1-\frac{1}{n^r}$, $\mu(G_{\omega_m})=0$ and $\mu(G_{\neg\omega_m})=1$. In particular, in the class of finitely generated abelian groups, there are definable sets with positive measure smaller than 1. 

Now we consider what happens when considering boolean combinations of sets of the previous form. Writing a formula as a finite disjunction of finite conjunctions of formulae, it is enough to compute the measure of all possible intersections of formulae.

Let $n_1,\ldots, n_k$ be positive natural numbers. Then $$G_{\delta_{n_1}}\cap\cdots\cap G_{\delta_{n_k}}=G_{\delta_{\mathrm{lcm}(n_1,\ldots, n_k)}},$$ where $\mathrm{lcm}(n_1,\ldots,n_k)$ is the least common multiple of $n_1,\ldots, n_k$. Moreover, for each positive natural numbers $n, m$, $\mu(G_{\delta_n}\cap G_{\omega_m})=\mu(\lbrace 0\rbrace)=0$ and $\mu(G_{\delta_n}\cap G_{\neg\omega_m})=\mu(G_{\delta_n})=\frac{1}{n^r}$.  

Now let $n_1,\ldots, n_k$ be positive natural numbers and consider the definable sets $S:=G_{\neg\delta_{n_1}\wedge\cdots\wedge\neg\delta_{n_k}}$. This set is the complement of the set $M:=G_{\delta_{n_1}\vee\cdots\vee\delta_{n_k}}=G_{\delta_{n_1}}\cup\cdots\cup G_{\delta_{n_k}}$. By the inclusion-exclusion principle (\cite[Lemma 21.3.1 (a)]{FJ}), the measure of $M$ is given by
\begin{align*}
\mu(M)&=\sum_{i=1}^k{(-1)^{i-1}\sum_{1\leq j_1<\cdots<j_i\leq k}{\mu(G_{\delta_{n_{j_1}}}\cap\cdots\cap G_{\delta_{n_{j_i}}})}}\\
      &=\sum_{i=1}^k{(-1)^{i-1}\sum_{1\leq j_1<\cdots<j_i\leq k}{\mu(G_{\delta_\mathrm{lcm}(n_{j_1},\ldots, n_{j_i})})}}\\
      &=\sum_{i=1}^k{(-1)^{i-1}\sum_{1\leq j_1<\cdots<j_i\leq k}{\frac{1}{\mathrm{lcm}(n_{j_1},\ldots, n_{j_i})^r}}}.
\end{align*}

It follows that
$$\mu(S)=1-\mu(M)=1+\sum_{i=1}^k{(-1)^{i}\sum_{1\leq j_1<\cdots<j_i\leq k}{\frac{1}{\mathrm{lcm}(n_{j_1},\ldots, n_{j_i})^r}}}.$$
Next, we compute by induction $\mu(G_{\delta_n}\cap G_{\neg\delta_{m_1}}\cap\cdots\cap G_{\neg\delta_{m_k}})$ for $k\ge 1$. 

If $k=1$, 
note that $G_{\delta_n}\cap~ G_{\neg\delta_m}=G_{\delta_n}\setminus G_{\delta_{\mathrm{lcm}(n,m)}}$, from which we conclude that $$\mu(G_{\delta_n}\cap G_{\neg\delta_m})=\mu(G_{\delta_n})-\mu(G_{\delta_{\mathrm{lcm}(n,m)}})=\frac{1}{n^r}-\frac{1}{\mathrm{lcm}(n,m)^r}.$$ 

Now let $k>1$. To simplify the notation, denote $G_{m_i}:=G_{\delta_n}\setminus G_{\delta_{\mathrm{lcm}(n,m_i)}}$ for every $i$. Assume by induction that, for every $i\le k$,

$$\mu(G_{{m_{j_1}}}\cap\cdots\cap G_{{m_{j_i}}})=\frac{1}{n^r}+\sum_{t=1}^i{(-1)^{t}\sum_{\substack{{j_1\leq u_1<\cdots<u_t\leq j_i}\\ u_l\in\lbrace j_1,\ldots,j_i\rbrace}}{\frac{1}{\mathrm{lcm}(n,m_{u_1},\ldots, m_{u_t})^r}}}$$

and denote this quantity by $I(j_1,\ldots,j_i)$. Furthermore, let $J(j_1,\ldots, j_i):=I(j_1,\ldots, j_i)-\frac{1}{n^r}$.

Moreover, it is not difficult to see that $$G_{{m_1}}\cup\cdots\cup G_{{m_{k+1}}}=G_{\delta_n}\setminus G_{\mathrm{lcm}(n,m_1,\ldots,m_{k+1})}$$ and therefore
   $$\mu(G_{{m_1}}\cup\cdots\cup G_{{m_{k+1}}})=\frac{1}{n^r}-\frac{1}{\mathrm{lcm}(n,m_1,\ldots, m_{k+1})^r}.$$ 

Write $M_{k+1}:=(-1)^{k} \mu(G_{\delta_n}\cap G_{\neg\delta_{m_1}}\cap\cdots\cap G_{\neg\delta_{m_{k+1}}})$.
Then, applying the inclusion-exclusion principle we obtain:
\begin{align*}
   M_{k+1}&=(-1)^{k} \mu( G_{{m_1}}\cap\cdots\cap G_{{m_{k+1}}})\\
   &=\mu(G_{{m_1}}\cup\cdots\cup G_{{m_{k+1}}})+\sum_{i=1}^{k}(-1)^{i}\sum_{1\leq j_1<\cdots<j_i\leq k+1}{\mu(G_{{m_{j_1}}}\cap\cdots\cap G_{{m_{j_i}}})}\\
   &=\frac{1}{n^r}-\frac{1}{\mathrm{lcm}(n,m_1,\ldots, m_{k+1})^r}+\sum_{i=1}^{k}(-1)^{i}\sum_{1\leq j_1<\cdots<j_i\leq k+1}{I(j_1,\ldots, j_i)}\\
   &=(-1)^k\frac{1}{n^r}-\frac{1}{\mathrm{lcm}(n,m_1,\ldots, m_{k+1})^r}+\sum_{i=1}^{k}(-1)^{i}\sum_{1\leq j_1<\cdots<j_i\leq k+1}{J(j_1,\ldots, j_i)}.
   \end{align*}

   Note that, in the last summand of the previous sum, for every $l\in\lbrace 1,\ldots, k+1\rbrace$, the coefficient of $\frac{1}{\mathrm{lcm}(n,m_{j_1},\ldots, m_{j_l})^r}$ is given by
   $$\sum_{s=l}^k{(-1)^{l+s}\binom{k+1-l}{s-l}}=\sum_{v=0}^{k-l}{(-1)^v\binom{k+1-l}{v}}=(-1)^{k-l}.$$

Therefore we obtain:  

  $$\mu(G_{{m_{1}}}\cap\cdots\cap G_{{m_{k+1}}})=\frac{1}{n^r}+\sum_{i=1}^{k+1}{(-1)^{i}\sum_{1\leq j_1<\cdots<j_i\leq k+1}{\frac{1}{\mathrm{lcm}(n,m_{j_1},\ldots, m_{j_i})^r}}},$$
  as claimed.

Finally, note that $$G_{\delta_{n_1}}\cap\cdots\cap G_{\delta_{n_l}}\cap G_{\neg\delta_{m_1}}\cap\cdots\cap G_{\neg\delta_{m_k}}=G_{\delta_{\mathrm{lcm}(n_1,\ldots, n_l)}}\cap G_{\neg\delta_{m_1}}\cap\cdots\cap G_{\neg\delta_{m_k}},$$ which is a case that we have already treated.

Putting everything together we obtain the following.

\begin{proposition}
\label{prop: measure of def sets in abelian groups}
    Let $\mathscr{A}$ be the class of torsion-free finitely generated abelian profinite groups and let $\mathcal{U}$ be the class of all formulae. Let $G\in\mathscr{A}$ have rank $r$ and let $\varphi$ be the formula
    $$\bigvee_{i=1}^s \left(\delta_{n^{(i)}_{1}}\wedge\cdots\wedge\delta_{n^{(i)}_{l_i}}\wedge\neg\delta_{m^{(i)}_{1}}\wedge\cdots\wedge\neg\delta_{m^{(i)}_{k_i}}\right)$$
    for some positive integer $s$ and some non-negative integers $l_i, k_i$ (at least one positive) for every $i\in\lbrace 1,\ldots, s\rbrace$. For each $i\in\lbrace 1,\ldots, s\rbrace$ denote 
    $$\varphi_i:=\delta_{n^{(i)}_{1}}\wedge\cdots\wedge\delta_{n^{(i)}_{l_i}}\wedge\neg\delta_{m^{(i)}_{1}}\wedge\cdots\wedge\neg\delta_{m^{(i)}_{k_i}}.$$
    Then, setting $n_i:=\mathrm{lcm}(n^{(i)}_{1},\ldots, n^{(i)}_{l})$,
    $$
    \mu(G_{\varphi_i})=
    \begin{cases}
    \frac{1}{n_i^r}+\sum_{s=1}^{k_i}{(-1)^{s}\sum_{1\leq j_1<\cdots<j_s\leq k_i}{\frac{1}{\mathrm{lcm}(n_i,m_{j_1},\ldots, m_{j_i})^r}}}\ \ &\text{if}\ \ l_i,k_i\neq 0,\\
    \frac{1}{n_i^r}\ \ &\text{if}\ \ k_i= 0,\\
    1+\sum_{s=1}^{k_i}{(-1)^{s}\sum_{1\leq j_1<\cdots<j_s\leq k_i}{\frac{1}{\mathrm{lcm}(m_{j_1},\ldots, m_{j_i})^r}}} \ \ &\text{if}\ \ l_i = 0,
    \end{cases}
    $$
    
    and $$\mu(G_\varphi)=\sum_{i=1}^s{(-1)^{i-1}\sum_{1\leq j_1<\cdots<j_i\leq s}{\mu(G_{\varphi_{j_1}}\cap\cdots\cap G_{\varphi_{j_i}}}}).$$
    In particular, $\mathrm{Spec}_\mu(\mathcal{U}, \mathscr{A})\subseteq \Q$.
\end{proposition}

\begin{remark}
    Note that the sets $G_{\varphi_{j_1}}\cap\cdots\cap G_{\varphi_{j_i}}$ are again intersections of the form $G_{\varphi_i}$ and therefore their measure can be computed using the inclusion-exclusion principle as before.
    
\end{remark}

\appendix

\section{Further results on \texorpdfstring{$\mathcal{F}$}{F}-spectra in profinite groups}
\label{sec: appendix A further results}

\subsection{Subsets of profinite groups of arbitrary measure}

All examples of definable sets of positive measure different from one presented in this paper have rational measure. Therefore the following is a natural question, that was posed to us by Steffen Kionke.

\begin{question}
\label{quest:irrational measure}
    Are there definable sets in profinite groups with irrational measure?
\end{question}

Even if we cannot answer Question \ref{quest:irrational measure}, we present here the construction of a countably-based profinite group that contains subsets of arbitrary measure $\alpha\in[0,1]$, inspired by an example in \cite[Page 2]{LP}.

Consider a countably-based profinite group and choose a filtration $G=G(0)>G(1)>G(2)>\ldots$ of open subgroups. Write $m_n = \frac{n(n+1)}{2}$ and $P_n=\lvert G(m_{n-1}) : G(m_n) \rvert.$ Hence, for every $g\in G$,
\[
  \mu(g G(m_n)) = \frac{1}{\prod_{i=1}^n P_i}.
\]
Fix a real $\alpha\in [0,1]$. Using the generalised Cantor expansion of real numbers, 
there is a sequence of integers $(d_n)_{n=1}^\infty$ such that $0\le d_n \le P_n-1$ and
\[
   \alpha = \sum_{n=1}^\infty \frac{d_n}{\prod_{i=1}^n P_i}.
\]
We now define a subset of $G$ of measure $\alpha$ in the following way.
\begin{enumerate}
    \item \textbf{Step 1:} Let $C_1$ be the union of exactly $d_1$ arbitrary cosets of $G(m_1)$ in $G$. Since $d_1 \le P_1 - 1$, there is at least one coset of $G(m_1)$ not in $C_1$.
    \item \textbf{Step $\boldsymbol{n}$:} Suppose that we defined subsets $C_1,\ldots,C_{n-1}$ of $G$ and that $C_i$ is a union of cosets of $G(m_i)$ in $G(m_{i-1})$. The complement $G \setminus \bigcup_{i=1}^{n-1} C_i$ is a union of cosets of $G(m_{n})$ in $G(m_{n-1})$. Define $C_n$ by arbitrarily choosing exactly $d_n$ of these available $G(m_n)$ cosets.
    \item Set $C = \bigcup_{n=1}^\infty C_n$. Because each coset of $G(m_n)$ has measure $\frac{1}{P_1 \cdots P_n}$, the total measure is exactly:$$\mu(C) = \sum_{n=1}^\infty \frac{d_n}{P_1 P_2 \cdots P_n} = \alpha.$$
\end{enumerate}

\subsection{\texorpdfstring{$\mathcal{F}$}{F}-spectra in free profinite groups of countable rank}

In this section we show that, if $F$ denotes the free profinite group of countable rank and $\mathcal{P}$ is the family of positive existential formulae, then $\mathrm{Spec}_\mu(F,\mathcal{P})=\lbrace 0,1\rbrace$.  

\begin{lemma}
\label{lem: definable sets in count free have measure 0-1}
Let $S$ be a closed subset of the free profinite group $F$ of countable rank. Then, either $S=F$ or $\mu(S)=0$. In particular, the same is true for the set $F_\varphi$ where $\varphi$ is a positive existential formula. 
\end{lemma}

\begin{proof}
Suppose that $S$ is a proper subset of $F$. Since $S$ is closed, there exists a finite quotient $Q$ of $F$ such that $\pi_Q(S)\subsetneq Q$ and thus $\mu_Q(\pi_Q(S)) <1$.

As $F$ has countable rank, for every integer $m$ there exists a surjective map $\pi_{Q^m}:F\to Q^m$ such that $\pi_i\circ \pi_{Q^m}=\pi_Q$, where $\pi_i: Q^m\to Q$ is the projection onto the $i$-th factor for all $i$. It follows that $\pi_{Q^m}(S)\subseteq(\pi_Q(S))^m$ and 
therefore $$\mu(S)\le \mu_{Q^m}(\pi_{Q^m}(S)) \le \mu_{Q}(\pi_{Q}(S))^m \to 0 $$
as $m$ tends to infinity. The claim about $F_\varphi$ follows from Lemma \ref{lem: sets defined by positive existential formulae are measurable}.
\end{proof}

\begin{corollary}
    Let $F$ be the free profinite group of countable rank.
    There exists no non-surjective word $w$ that has finite width in $F$ and such that $w(F)$ has finite index in $F$. 
\end{corollary}

\subsection{An \texorpdfstring{$\mathcal{F}$}{F}-spectrum containing \texorpdfstring{$\Q\cap[0,1]$}{Q[0,1]}}

In this short section we prove the following.

\begin{lemma}
Denote by $\mathcal{EP}$ the family of existential formulae with parameters and by $\mathscr{F}$ the family of free profinite groups of finite rank. Then $\mathbb Q\cap [0,1]\subseteq \mathrm{Spec}_\mu(\mathcal{EP}, \mathscr{F})$.
\end{lemma}

\begin{proof}
Let $m$ and $r$ be positive integers, let $\widehat{F}_r$ be the free profinite group of rank $r$ and consider the short exact sequence 
$$1\to K_m\to \widehat{F}_r \to {\mathbb Z}^r/m{\mathbb Z}^r \to 1.$$
The group $K_m$, whose measure is $m^{-r}$, is given by $[\widehat{F}_r,\widehat{F}_r]\widehat{F}_r^m$. Since the commutator word has finite width in the finitely generated profinite group $\widehat{F}_r$, the group $K_m$ is definable in $\widehat{F}_r$.

Choose $t\le m^r$ distinct cosets $c_1 K_m,\ldots,c_t K_m$. Then there is an existential formula with parameters $c_1,\ldots, c_t$
defining the union of these cosets, whose measure is $\frac{t}{m^r}$. Given any rational number $\frac{a}{b}\in[0,1]$, take $m=b$ and $t=ab^{r-1}$. Then
\[
\frac{t}{m^r} = \frac{ab^{r-1}}{b^r} =\frac{a}{b}.
\]
\end{proof}

\section{A family of finite \texorpdfstring{$p$}{p}-groups where \texorpdfstring{$x^p$}{xp} has positive measure converging to zero}
\label{sec: family of p-gps where x^p has positive measure}

It is easy to see that the measure of the image of the word $x^p$ is positive in any cyclic group $C_{p^k}$ with $k>1$ (compare with Section \ref{sec: measures in fg abelian groups}). However, this measure is constantly equal to $\frac{1}{p}$. In this section, we find a family of finite $p$-groups where the measure of the image of the $p$-power map is positive but converging to zero, thus yielding an alternative proof of Theorem \ref{prop:A} in this special case.

For this purpose, we might think of the family $C_p \wr C_{p^k}$ ($k\ge 1$) as a natural candidate, as the intuition would indicate that the $p$-powers `die out' when $k$ tends to infinity. Surprisingly, this is not the case (as we show in Remark~\ref{rem:measure_x_p}). However, a slight modification of the groups above does the trick.

\subsection{Wreath products}
We recall here the necessary technical details about the wreath products that we will consider. 

Given a finite $p$-group $G$, we can form the wreath product $C_p \wr G$ with respect to the regular action of $G$ on itself by left-multiplication. In this special case, this wreath product can be identified with $\mathbb{F}_p[G] \rtimes G$, where $\mathbb{F}_p[G]$ is the $\mathbb{F}_p$\nobreakdash-group algebra of $G$ on which $G$ acts by left-multiplication; that is, for $u\in \mathbb{F}_p[G]$ and $g\in G$ we define $u^{g^{-1}}=g\cdot u$. For $u,v\in \mathbb{F}_p[G]$ and $g_1,g_2\in G$ the multiplication in $C_p \wr G$ works as follows

\begin{equation}
\label{eq: formula product in Cp wreath G}
u g_1 \cdot v g_2 = u (g_1 v g_1^{-1}) g_1 g_2= [u+ g_1 \cdot v ]g_1 g_2.
\end{equation}

Since we will need this later, we observe that it readily follows from (\ref{eq: formula product in Cp wreath G}) that the exponent of $C_p \wr G$ is at most $p \exp(G)$. Moreover, it is easy to see that, if $G$ is $d$-generated, then $C_p\wr G$ can be generated by at most $d+1$-elements.

We will now describe the behaviour of the $p$-power map in $C_p\wr G$. Let $x= u g\in C_p \wr G = \mathbb{F}_p[G] \rtimes G$, then
\begin{equation}\label{eq:p_power}
  x^p = [u + g \cdot u + \ldots + g^{p-1} \cdot u]g^p =: T_g(u)g^p, 
\end{equation}
where $T_g:\mathbb{F}_p[G] \to \mathbb{F}_p[G]$ denotes the linear map $u\mapsto (1+g+\ldots+g^{p-1})\cdot u$. 

\begin{lemma}\label{lem:binom}
 Let $G$ be a finite $p$-group. For every $g\in G$ the linear map $T_g$ satisfies $T_g(u) = (g-1)^{p-1} \cdot u$ for all $u\in \mathbb{F}_p[G]$.
 
 \begin{proof}
 Using Newton's binomial formula, 
 \[
   (g-1)^{p-1} = \sum_{i=0}^{p-1} \binom{p-1}{i} g^i (-1)^{p-1-i}.
 \] 
 Since the coefficients $\binom{p-1}{i}$ are congruent to $(-1)^i\mod p$, the coefficient of $g^i$ in the previous sum is $(-1)^i \cdot(-1)^{p-1-i} = (-1)^{p-1}=1$ for $p>2$. For $p=2$, we get that $(-1)^{2-1} = -1 \equiv 1 \mod 2$. In any case, the claimed equality follows.
 \end{proof}
\end{lemma}

\subsection{The family}
Here we investigate a family of finite $p$-groups with few $p$-powers.

Let $G_k := C_p \wr W_k$, $W_k := C_p\wr C_{p^k}$, both wreath products in regular action. So $G_k = \mathbb{F}_p[W_k] \rtimes W_k$ and $W_k = \mathbb{F}_p[C_{p^k}]\rtimes C_{p^k}$. Note that $G_k$ is generated by the elements $1_{\mathbb{F}_p[W_k]}, 1_{\mathbb{F}_p[C_{p^k}]}$ and the generator $g$ of $C_{p^k}$. The orders of the groups are
\[
  \lvert W_k \rvert = p^{k+p^k}=: N_k \text{\quad and \quad} \lvert G_k \rvert = N_k p^{N_k}.
\]

Recalling that, for every $g\in W_k$, the map $T_g:\mathbb{F}_p[W_k] \to \mathbb{F}_p[W_k]$ is the linear map $u\mapsto (1+g+\ldots+g^{p-1})\cdot u$ and that, by \eqref{eq:p_power}, $(ug)^p=T_g(u)g^p$, the size of the image of the $p$-power map $w_p:G_k\rightarrow G_k$ is bounded from above as follows: 
\begin{equation}\label{eq:image_bound}
   \lvert \mathrm{im}(w_p)\rvert < \sum_{h\in W_k} \lvert \mathrm{im}(T_h) \rvert .
\end{equation}

Note that the inequality \eqref{eq:image_bound} is strict, because there are distinct elements $h_1,h_2\in W_k$ such that $h_1^p=h_2^p$ and a sharper bound could certainly be obtained using the inclusion-exclusion formula. However, the above bound is sufficient for our purposes. We will now bound $ \lvert \mathrm{im}(T_h) \rvert$ from above.

If $h=1$, it is clear that $\lvert \mathrm{im}(T_1)\rvert=1$.

If $h\neq 1$, let us write $o(h)$ for the order of $h\in W_k$. Then, the group-algebra $\mathbb{F}_p[W_k]$ splits into $N_k/o(h)$ copies of the regular representation $\mathbb{F}_p[\langle h\rangle]$ of $\langle h\rangle$.

Let us focus on $A:=\mathbb{F}_p[\langle h\rangle]$. By Lemma~\ref{lem:binom}, we can reduce ourselves to study the linear map $f:A\to A$ given by $x\mapsto (h-1)\cdot x$. It can be easily seen by direct computation with the basis $\{1,h,\ldots,h^{o(h)-1}\}$ that $\dim\ker f^\ell=\ell$ for any $0\le \ell\le o(h)$. It follows that $\dim\ker T_h =p-1$.

Thus, on $\mathbb{F}_p[W_k]$ the dimension of the kernel of $T_h$ is $\frac{N_k}{o(h)} (p-1)$ and its image has dimension $$N_k-\frac{N_k}{o(h)} (p-1) = N_k \left(1 - \frac{p-1}{o(h)}\right).$$

Now, the exponent of $W_k$ is exactly $p^{k+1}$, hence the image of $T_h$ has dimension at most $N_k \left(1 - \frac{p-1}{p^{k+1}}\right)$ for all $h\in W_k$ and for any $1\neq h\in W_k$
\[
  \lvert \mathrm{im}(T_h) \rvert \le p^{N_k \left(1 - \frac{p-1}{p^{k+1}}\right)}
\]
and the image of the $p$-power map is of size at most $N_k \cdot p^{N_k \left(1 - \frac{p-1}{p^{k+1}}\right)}$. Then the probability that an element of $G_k$ is a $p$-power is at most
\[
  \frac{N_k \cdot p^{N_k (1 - \frac{p-1}{p^{k+1}})}}{N_k \cdot p^{N_k}} = p^{-N_k \frac{p-1}{p^{k+1}}} = p^{-(p-1)p^{p^k-1}}
\] 
which tends to zero for $k\to \infty$.

\begin{remark}\label{rem:measure_x_p}
    The measure of the image of the $p$-power map $w_p$ on $W_k=C_p\wr C_{p^k}$ is bounded away from zero. In fact, if $C_{p^k}=\langle g\rangle$, for all distinct elements $x\neq x'$ in the set 
    $$X= \{g^i\mid 0 \le i\le p^{k-1}-1\},$$ one has $x^p\neq (x')^p$. Moreover, $\lvert X\rvert={p^{k-1}}$. Thus
    \[
      \lvert\mathrm{im}(w_p)\rvert \ge \sum_{x\in X} \lvert \mathrm{im}(T_x)\rvert.
    \]
    Observe that the order $o(x)$ for $1\neq x\in X$ is at least $p$, so that in the above estimates (rewritten for $C_p\wr C_{p^k}$) the image of $T_x$ has size \emph{at least} $p^{p^k-1+1/p}$. Thus,
    \[
       \lvert\mathrm{im}(w_p)\rvert \ge (p^{k-1}-1) \cdot p^{p^k-1+1/p}
    \]
    and dividing by $\lvert W_k\rvert = p^{k+p^k}$ we obtain that the measure of the image is bounded away from $0$ for all $k\in \mathbb N$.
\end{remark}

\subsection{The word \texorpdfstring{$x^p[y,z]$}{xp[y,z]}}

Recall from the Introduction that it would be interesting to calculate the measure $\mu(\mathrm{im}(x^p[y,z]))$ in the free pro-$p$ group on $3$ generators. We might hope that an argument similar to that used for the word $x^p$ could lead us to conclude that this measure is zero. This is not the case, and indeed, in this subsection we show that the measure of the image of the word $x_1^p[x_2,x_3]$ is bounded away from zero on $G_k$. 

A direct computation, combining equation \eqref{eq:p_power}, for $x,y,z\in W_k$ and $u,v,t\in B=\mathbb{F}_p[W_k]$ yields:
\begin{align*}
  (ux)^p [vy,tz] &= \left[T_x(u) + x^p(zy)^{-1} \cdot ( (y-1) \cdot t - (z-1)\cdot v) \right] \, x^p[y,z] \\ & :=\omega_{x,y,z}(u,v,t) \, x^p[y,z].
\end{align*}

We first observe that, if $\langle y,z\rangle =W_k$, then, since the augmentation ideal $I(B)$ of $B$ is generated by $y-1$ and $z-1$, we have that 
\[
   \omega_{x,y,z}(B,B,B) \supseteq (y-1)\cdot B + (z-1) \cdot B = I(B)
\]
so that the size of $\mathrm{im}(\omega_{x,y,z})$ is at least $p^{N_k-1}$. 

In a similar fashion to the previous subsection, to bound the image from below we need to pick elements $x,y,z\in W_k$ that give different images in $G_k$ via the word $x^p[y,z]$. First observe that, if $x_1^p x_2^{-p} \notin \mathbb{F}_p[C_{p^k}]$ then $x_1^p[y_1,z_1] \neq x_2^p[y_2,z_2]$ (project modulo the base group of $W_k$). Therefore the set $X=\{g^i\mid 1\le i \le p^{k-1}-1\}$ is of size $p^{k-1}$ and it is such that for any $x_1\neq x_2\in X$ and $y_1,y_2,z_1,z_2 \in W_k$ we have that $x_1^p[y_1,z_1] \neq x_2^p[y_2,z_2]$.

On the other hand, it is clear that the element $a=1_{\mathbb{F}_p[C_{p^k}]}$ together with the generator $g$ of $C_{p^k}$ generate $W_k$. Now, for any element $v\in I(\mathbb{F}_p[C_{p^k}])$, since the augmentation ideal is contained in the Frattini, the elements $va,vg$ still generate $W_k$ and
\[
   [va,vg] = (g^{-1}-1)\cdot (v+a).
\]
We now observe that the multiplication by $g^{-1}-1$ induces a linear map on $\mathbb F_p[C_{p^k}]$ whose kernel is $Z(W_k)$ and $\dim_{\mathbb{F}_p} Z(W_k)=1$. Hence, if we define the set $$\overline{Y}=\{ (va,vg) \mid v\in I(\mathbb{F}_p[C_{p^k}])\}/Z(W_k),$$ then $\lvert \overline{Y} \rvert=p^{p^k-2}$. Finally, choose a set of representatives $Y$ for $\overline{Y}$ in $\mathbb{F}_p[C_{p^k}]$. 

By construction, it follows that the images of the cosets $(Bx,(By,Bz))$ for $x\in X$ and $(y,z)\in Y$ via the word $x^p[y,z]$ are all pairwise disjoint, because they have different images modulo $B$.

Putting all our estimates together
\begin{align*}
   \lvert \mathrm{im}(a^p[b,c])\rvert & \ge \sum_{x\in X,\ (y,z)\in Y} \lvert \mathrm{im}(\omega_{x,y,z})\rvert  \ge p^{k-1} \cdot p^{p^k-2} \cdot p^{N_k-1}  \\
  & = \frac{N_k p^{N_k}}{p^4} = \frac{\lvert G_k\rvert}{p^4} 
\end{align*}
Thus, $\lvert \mathrm{im}(x_1^p[x_2,x_3])\rvert/\lvert G_k \rvert \ge 1/p^4>0$ for all $k\in \mathbb N$.

\end{document}